\documentclass[11pt,reqno]{amsart}

\usepackage[foot]{amsaddr}
\usepackage[utf8]{inputenc}
\usepackage{color}
\usepackage[dvipsnames]{xcolor} 

\definecolor{dred}{rgb}{.8,0.2,.2}
\definecolor{ddred}{rgb}{.8,0.5,.5}
\definecolor{dblue}{rgb}{.2,0.2,.5}

\definecolor{myurlcolor}{rgb}{0.6,0,0}
\definecolor{mycitecolor}{rgb}{0,0,0.8}
\definecolor{myrefcolor}{rgb}{0,0,0.8}
\usepackage{hyperref}
\hypersetup{colorlinks,
linkcolor=myrefcolor,
citecolor=mycitecolor,
urlcolor=myurlcolor}

\usepackage{graphicx,psfrag}
\usepackage{amssymb,amsmath,amsthm}
\usepackage{bm,mathrsfs,accents}
\usepackage{enumerate}
\usepackage{a4wide}
\usepackage{hyperref}

\usepackage{tikz-cd}

\usepackage{caption}
\usepackage{subcaption}
\usepackage{longtable}

\usepackage{tcolorbox,pgfplots}

\usepackage{ifthen}
\usepackage[dvipsnames]{xcolor} 
\newboolean{show_comments}
\setboolean{show_comments}{true} 

\DeclareRobustCommand{\todo}[1]{
\ifthenelse{\boolean{show_comments}}
{\begingroup\color{dred}{[\textbf{TODO:} #1]}\endgroup}
{}
}

\DeclareRobustCommand{\nick}[1]{
\ifthenelse{\boolean{show_comments}}
{\begingroup\color{Blue}{[\textbf{Nick:} #1]}\endgroup}
{}
}

\DeclareRobustCommand{\francesco}[1]{
\ifthenelse{\boolean{show_comments}}
{\begingroup\color{Green}{[\textbf{Francesco} #1]}\endgroup}
{}
}

\numberwithin{equation}{section}
\allowdisplaybreaks[4]

\newtheorem{proposition}{Proposition}[section]
\newtheorem{theorem}[proposition]{Theorem}
\newtheorem{lemma}[proposition]{Lemma}
\newtheorem{corollary}[proposition]{Corollary}
\theoremstyle{definition}
\newtheorem{definition}[proposition]{Definition}
\newtheorem{remark}[proposition]{Remark}

\newtheorem{example}[proposition]{Example}

\newtheorem*{details}{Details}
\expandafter\def\expandafter\details\expandafter{\details\small}

\newenvironment{proofof}[1]{\smallskip\noindent{\emph{Proof~of~#1.}}%
  \hspace{1pt}}{\hspace{-5pt}{\nobreak\quad\nobreak\hfill\nobreak%
    $\square$\vspace{2pt}\par}\smallskip\goodbreak}

\newcommand{\spt}{\mathop{\rm spt}}

\newcommand{\co}{\mathop{\rm co}}

\newcommand{\tr}{\mathop{\rm tr}}
\newcommand{\id}{\mathop{\rm id}}
\newcommand{\dom}{\mathop{\rm dom}}

\newcommand{\lip}{\mathop{\rm Lip}}

\newcommand{\esssup}{\mathop{\rm ess\,sup}}
\newcommand{\T}{\intercal}

\newcommand{\Lip}{\mathrm{Lip}}
\newcommand{\R}{\mathbb{R}}

\makeatletter
\newcommand{\opnorm}{\@ifstar\@opnorms\@opnorm}
\newcommand{\@opnorms}[1]{%
  \left|\mkern-1.5mu\left|\mkern-1.5mu\left|
   #1
  \right|\mkern-1.5mu\right|\mkern-1.5mu\right|
}
\newcommand{\@opnorm}[2][]{%
  \mathopen{#1|\mkern-1.5mu#1|\mkern-1.5mu#1|}
  #2
  \mathclose{#1|\mkern-1.5mu#1|\mkern-1.5mu#1|}
}
\newcommand{\eps}{\varepsilon}
\makeatother

\usepackage[
backend=biber,
natbib=true,
maxbibnames=4,
style=numeric-comp,
giveninits=true,
eprint=false,
url=false,
doi=false,
isbn=false,
alldates=year,
urldate=short,
sorting=nyt]{biblatex}
\AtEveryBibitem{\clearfield{note}}
\AtEveryCitekey{\clearfield{note}}
\AtEveryBibitem{\clearfield{doi}}
\AtEveryBibitem{\clearlist{language}}

\AtBeginBibliography{\footnotesize}

\usepackage{xpatch}

\xpatchbibmacro{volume+number+eid}{%
  \setunit*{\adddot}%
}{%
}{}{}

\DeclareFieldFormat[article]{number}{\mkbibparens{#1}}

\begin{document}

\title[Trajectory stabilization of nonlocal continuity equations]{Trajectory stabilization of nonlocal continuity equations by localized controls}

\author{Nikolay Pogodaev${}^\dagger$}
\address{${}^\dagger$Dipartimento di Matematica, Università degli Studi di Padova}
\email{nikolay.pogodaev@unipd.it}
\author{Francesco Rossi${}^\ddagger$}
\address{${}^\ddagger$Dipartimento di Culture del Progetto, Universtà Iuav di Venezia. The Author is a member of G.N.A.M.P.A. (I.N.d.A.M.).}
\email{francesco.rossi@iuav.it}


\begin{abstract} We discuss stabilization around trajectories of the continuity equation with nonlocal vector fields, where the control is localized, i.e., it acts on a fixed subset of the configuration space.

  We first show that the correct definition of stabilization is the following: given an initial error of order $\eps$, measured in Wasserstein distance, one can improve the final error to an order \( \varepsilon^{1+\kappa} \) with $\kappa>0$.

  We then prove the main result: assuming that the trajectory crosses the subset of control action, stabilization can be achieved. The key problem lies in regularity issues: the reference trajectory needs to be absolutely continuous, while the initial state to be stabilized needs to be realized by a small Lipschitz perturbation or being in a very small neighborhood of it.

\end{abstract}

\maketitle

\section{Introduction}

In recent years, the study of systems describing a crowd of interacting autonomous agents
has drawn a great interest from the mathematical and control communities.
A better understanding of such interaction phenomena can have a strong impact in several key applications, such as road traffic and egress problems for pedestrians. For a few reviews about this topic, see e.g.
\cite{axelrod,camazine,CPTbook,helbing,SepulchreReview}.

Clearly, there is a wealth of different mathematical models available to describe crowds. In this article, we use one of the most popular methods in the mathematical community: the continuity equation with non-local velocity, that is,
\begin{equation}
  \label{eq:conteq-uncontrolled}
\partial_t\mu_t+\nabla\cdot \left(\left(V_t(x,\mu_t)\right)\mu_t\right) = 0.
\end{equation}
Here, the crowd is described by a measure $\mu_t$ evolving in time, according to the action of the vector field $V_t(x,\mu_t)$. The key feature of the model is exactly non-locality: the vector field $V$ at point $x$ depends on the whole distribution $\mu_t$, not only on the value of its density at $x$. Analysis of such equation is by now well-established, together with efficient numerical methods and issues related to singular limits, see e.g. \cite{piccoliTransportEquationNonlocal2013,AGS05,Coclite20231205,Keimer2023183,Camilli20187213}.

Beside the description of interactions, it is now relevant to study problems of crowd control, i.e., of controlling such systems by acting on few agents, or on a small subset of the configuration space.
Roughly speaking, basic problems for such models include controllability (i.e., reaching a desired configuration), optimal control (i.e., the minimization of a given functional) and stabilization (i.e., counteract perturbations to stay around a given configuration/trajectory). Many results for different contexts can be found in \cite{blaq,PRT15,DMRcontrol,DMRmin,FS,carmona,CPRT17,bullo,achdou1,CHM,fornasierMeanfieldOptimalControl2018,Chertovskih2023,Ciampa2021185}.

The standard setup for a control system in crowd modeling is as follows:
\begin{equation}
  \label{eq:conteq}
\partial_t\mu_t+\nabla\cdot \left(\left(V_t(x,\mu_t)+u_t(x)\right)\mu_t\right) = 0,
\end{equation}
where \( u_t \) is an external vector field with additive action.

In this article, we focus on the problem of \emph{stabilization around trajectories}. A rough description of the problem is as follows:
\begin{tcolorbox}
Given a reference trajectory $\mu_t$ of \eqref{eq:conteq-uncontrolled} on a time interval $[0,T]$ and a measure $\varrho$, which is \emph{close} to the initial state \( \mu_0=\varrho_{0} \), can we find an admissible control $u$ such that the corresponding solution \( \mu^{u}_t \) of \eqref{eq:conteq} issuing from \( \varrho \) comes \emph{very close} to $\mu_t$, at time \( T \)?
\end{tcolorbox}
The intuitive concept of stabilization is illustrated in Figure~\ref{fig:paths}, while writing its rigorous definition requires some work: we need to give, at least, a precise meaning to the words \emph{close} and \emph{very close}. Moreover, the class of admissible controls need to be precisely set as well.

\begin{figure}[htb]
  \begin{center}
    \includegraphics[width=0.65\textwidth]{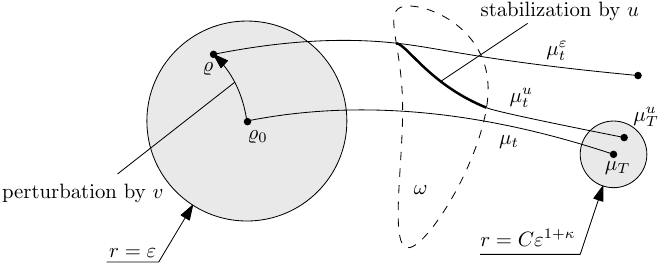}
    \caption{
      The concept of stabilization: we perturb \( \varrho_0 \) by a Lipschitz control \( v \). Our goal is to neutralize this perturbation with another Lipschitz control acting inside \( \omega \) only.
    }
\label{fig:paths}
\end{center}
\end{figure}

We start by choosing a class of admissible controls.
In the present article, this role will be played by \emph{Lipschitz localized} controls, i.e., nonautonomous vector fields that are globally Lipschitz on the \( x \)-space and which support is contained in a fixed non-empty open subset \( \omega \) of \( \mathbb{R}^d \).
More precisely, the family of \emph{admissible controls} is given by
\begin{equation}
  \label{eq:CL}
u\in L^\infty\left([0,T];\Lip(\mathbb{R}^d;\mathbb{R}^{d})\right),\quad \spt(u_{t})\subset \omega,\quad t\in [0,T].
\end{equation}
The corresponding solutions of~\eqref{eq:conteq} will be called \emph{admissible trajectories}.

The nonlocal vector field \( V \) will have similar regularity properties, i.e, global Lipschitz continuity both in \( x \) and \( \mu \), which will be specified in Assumption \( (A_1) \) below.
Such regularity assumption for both $V$ and $u$ is very natural, since it ensures existence and uniqueness of solutions for~\eqref{eq:conteq}, see \cite{ambrosioTransportEquationCauchy2008}.
Moreover, it provides an ``explicit'' formula for the solution: 
\begin{equation}
\label{e-sol}
\mu_t=\Phi_{0,t\sharp}\varrho_{0},
\end{equation}
where \( \Phi_{0,t} \) is the flow of \( V+u \) and \( \sharp \) denotes the pushforward operator (see Section~\ref{sec:wass} below).

The advantage of this formulation is counterbalanced by the very fact that the solution is given by the pushforward of the flow. Indeed, the flow is an homeomorphism, then \eqref{e-sol} shows a form of rigidity in moving from $\varrho_0$ to $\mu_t$. For example, it has the following consequences:
\begin{enumerate}[(a)]
\item The sets \( \spt (\varrho_0) \) and \( \spt(\mu_t) \) are homeomorphic (e.g. one cannot create or destroy holes).
    \item If \( \spt(\varrho_0) \) is compact, then \( \spt (\mu_t) \) is also compact.
    \item If \( \varrho_0 \) has atoms, then \( \mu_t \) has the same atoms, usually  displaced (atoms cannot be split, merged, or removed).
\end{enumerate}
These observations already show that exact stabilization (i.e., exactly reaching the reference trajectory) is usually not possible, as it would require a homeomorphism between initial states $\varrho_0$ and $\varrho$. 
A similar discussion can be found in \cite{DMRcontrol} about exact controllability.
As a possible solution, in this article we will only focus on approximate stabilization, i.e., on getting \emph{very close} to $\mu_t$.
However, the ability to deform \( \varrho \) into \( \varrho_0 \) will play an important part in our discussion.

The second feature of admissible controls is the localization property, i.e., the fact that $\spt(u_t)\subset \omega$.
Clearly, in this context, one can act on a portion of mass only if it crosses $\omega$ at some time.
This condition will be made precise in Assumption $(A_2)$ below.


We now focus on the state space, i.e., the space in which the solutions of \eqref{eq:conteq} are defined.
Here our main choice will be \( \mathcal{P}_c(\mathbb{R}^d) \), the space of Borel probability measures on \( \mathbb{R}^d \) having compact supports.
The restriction to compactly supported measures is reasonable from the modeling point of view. Occasionally, we will restrict the state space even more, by considering only the space \( \mathcal{P}_c^{ac}(\mathbb{R}^d) \) of absolutely continuous measures in \( \mathcal{P}_c(\mathbb{R}^d) \). Moreover, both \( \mathcal{P}_c(\mathbb{R}^d) \) and  \( \mathcal{P}_c^{ac}(\mathbb{R}^d) \) are invariant sets of~\eqref{eq:conteq}, due to the representation~\eqref{e-sol}.

Given the state space, we need to endow it with a topology, that is necessary for dealing with approximate stabilization.
Moreover, we need to quantify stabilization, to make sense of the words ``near'' and ``very near''.
Our choice is to endow $\mathcal{P}_c(\R^d)$ with the (quadratic) Wasserstein distance $\mathcal{W}_2$, that is associated with optimal transportation \cite{villaniTopicsOptimalTransportation2003,zbMATH05306371}.
It is now very clear that such distance is very suitable to describe solutions of continuity equations with non-local velocities \cite{A04,ambrosioHamiltonianODEsWasserstein2008,PR13,gigliGeometrySpaceProbability2004}, eventually with control \cite{PRT15,DMRcontrol,DMRmin}.
We will recall the definition of Wasserstein distance in Section \ref{sec:wass}.
We just mention here that on compact sets it metrizes the weak convergence of measures, which is the natural topology for distributional solutions of \eqref{eq:conteq}.
We are now ready to give a proper definition of $\kappa$-stabilization.

  \begin{definition}  Let $A\subset \mathcal{P}_c (\mathbb{R}^d)$ be a set of measures, considered as initial conditions for \eqref{eq:conteq}. Let the trajectory \( \mu_t \) of~\eqref{eq:conteq} on the time interval $[0,T]$ corresponding to the initial condition \( \mu_0=\varrho_0 \) and the zero control \( u=0 \) be called \emph{the reference trajectory}.

    We say that the set $A\subset \mathcal{P}_c(\mathbb{R}^d)$ is \( \kappa \)-\emph{stabilized around the reference trajectory \( \mu_t \)} of \eqref{eq:conteq} if there exists \( C>0 \) such that for any \( \varepsilon>0 \) and $\varrho\in A$ with \( \mathcal{W}_2(\varrho,\varrho_0)< \varepsilon \) one can find an admissible control $u$ such that the corresponding trajectory $\mu_t^u$ of \eqref{eq:conteq} starting from $\varrho$ satisfies
  \begin{equation}
    \label{eq:stabilization}
     \mathcal{W}_2\left(\mu^u_{T},\mu_{T}\right)< C\varepsilon^{1+\kappa}.
   \end{equation}
   In other words, whatever \( \varepsilon \) we choose, any point of the set \( A\cap{\bf B}_{\varepsilon}(\varrho_0) \) can be steered into the ball \( {\bf B}_{C\varepsilon^{1+\kappa}}(\mu_{T}) \) by an admissible control, see~Fig~\ref{fig:paths}.
   Here \( {\bf B}_\varepsilon(\varrho_0) \) denotes the open Wasserstein ball of radius \( \varepsilon \) centered at \( \varrho_0 \).
 \end{definition}

 \begin{remark}
   Here we collect several known stabilization results.
   \begin{enumerate}[1.]
     \item The whole space \( \mathcal{P}_c(\mathbb{R}^d) \) is trivially \( 0 \)-stabilized, i.e., an initial error of order \( \varepsilon \) keeps being of the same order.
    It is indeed sufficient to choose \( u=0 \) and apply Proposition~\ref{prop:basic} below.

   \item In two particular cases some information about exact stabilization can be derived from the controllability results we presented in~\cite{DMRcontrol}.
   \begin{description}
     \item[Case 1] Let \( V \) only depend on \( x \) and let \( \varrho_0\in \mathcal{P}_c^{ac}(\mathbb{R}^d) \).
      Let Assumptions \( (A_{1,2}) \) below hold.
      Then, there exist a dense subset \( A \) of \( \mathcal{P}_c(\mathbb{R}^d) \) and \( \varepsilon_{*}>0 \) such that \( A\cap {\bf B}_{\varepsilon_{*}}(\varrho_0) \) can be approximately stabilized with order $\kappa$ for any $\kappa>0$.

     \item[Case 2] Let \( u \) act on the whole space, i.e., \( \omega=\mathbb{R}^d \), and \( \varrho_0\in \mathcal{P}^{ac}_c(\mathbb{R}^d) \).
      Let Assumptions \( (A_{1,2}) \) below hold.
    Also in this case, there exists a dense subset \( A\) of \( \mathcal{P}_c(\mathbb{R}^d) \) that can be approximately stabilized with order $\kappa$ for any $\kappa>0$.
   \end{description}
   \end{enumerate}
 \end{remark}

 Note that in the first case, we dropped the \emph{nonlocality assumption} for \( V \), while in the second case, we dropped the \emph{localization assumption} for \( u \).
 Combining both these assumptions poses a real challenge, as they are somewhat contradictory.
 Indeed, if \( V \) is nonlocal (i.e., depending on \( \mu \)), then any change in the mass inside \( \omega \) will affect the mass that has already crossed \( \omega \).
 But we have no option to counterbalance this effect, because \( u \) can act inside \( \omega \) only! In particular, this is the main reason why techniques from~\cite{DMRcontrol,DMRmin} cannot be applied here.


We now describe sets that can be \( \kappa \)-stabilized.
One may hope that the ball \( A={\bf B}_{\varepsilon_{*}}(\varrho_0) \) has this property, as soon as \( \varepsilon_{*} \) is small enough.
However, as we shall see in Section~\ref{s-sharp}, this is not true.
The reason is that \( {\bf B}_{\varepsilon_{*}}(\varrho_0) \) contains atomic measures.
A second hope would then to deal with $A={\bf B}_{\varepsilon_{*}}(\varrho_0)\cap\mathcal{P}^{ac}_c(\R^d)$, i.e., to impose the absolute continuity assumption.
Also in this setting, $\kappa$-stabilization is not achieved.
Indeed, here \( A \) contains absolutely continuous measures that are arbitrarily close to atomic measures, hence it is arbitrarily hard to send them close to the reference trajectory.

In order to proceed, we consider the problem with a different approach.
We start by perturbing the initial measure by a Lipschitz vector field \( v \) during the time interval \( [0,\varepsilon_0] \).
Can we neutralize this perturbation by applying a Lipschitz localized control?
To study this question, we define, for any \( \varepsilon_*>0 \), the following set:
\begin{equation}
  \label{eq:Gamma}
  \Gamma_{\varepsilon_*}(\varrho_0):= \left\{
      \Phi^v_{0,t\sharp}\varrho_0\,\colon \Phi^v \text{ is the flow of } v\in L^{\infty}\!\left([0,\varepsilon_*]; \Lip(\mathbb{R}^{d};\mathbb{R}^d)\right)\!
    \text{ with } \|v\|_{\infty}\le 1
     \right\}.
   \end{equation}
   Here, we endow \( \Lip(\mathbb{R}^d;\mathbb{R}^d) \) with the standard Lipschitz norm \( \|f\|_{\lip} = \max\{\|f\|_{\infty},\lip(f)\} \) making it a Banach space; \( \|f\|_{\infty} \) stands for the standard \( \sup  \)-norm and \( \lip(f) \) denotes the minimal Lipschitz constant of \( f \).
The norm of an admissible control \( v \) is defined by
\[
\|v\|_{\infty} := \esssup_{t\in[0,T]}\|v_t\|_{\lip}.
\]

Basically, the idea is to take a Lipschitz control \( v \), whose norm is at most \( 1 \), and use it to deform $\varrho_0$ on the time interval $[0,\varepsilon_*]$.
The properties of flows recalled above ensure that the set is contained in the closed Wasserstein ball $\overline{\bf B}_{\varepsilon_*}(\varrho_0)$.
We will show that \(\Gamma_{\varepsilon_*}(\varrho_0)\) is $\kappa$-stabilizable, if \( \varepsilon_* \) is sufficiently small.

\subsection{Main result} In this section, we state the key assumptions and the main results.

We first  precisely state the assumptions imposed on the vector field \( V \), the initial measure \(\varrho_0\) and on the action set \( \omega \).

\begin{tcolorbox}
\textbf{Assumption \((A_1)\):} The nonlocal time-dependent vector field  \( V\colon [0,T]\times \mathbb{R}^d \times \mathcal{P}_2(\mathbb{R}^{d})\to \mathbb{R}^d \) satisfies
\begin{itemize}
	\item \( V \) is measurable in \( t \);
	\item \( V \) is bounded and Lipschitz in \( x \) and \( \mu \), i.e., there exists \( M > 0 \) such that
	    \begin{gather*}
        \left|V_t(x,\mu)\right|\le M,\quad
	\left|V_{t}(x,\mu)-V_{t}(x',\mu')\right|\le M\left(|x-x'|+ \mathcal{W}_{2}(\mu,\mu')\right),
	  \end{gather*}
		for all \(x,x'\in \mathbb{R}^d \), \( \mu,\mu'\in \mathcal{P}_2(\mathbb{R}^d) \), \(t\in [0,T] \).
\end{itemize}
\end{tcolorbox}
This assumption, which dates back at least to~\cite{piccoliTransportEquationNonlocal2013}, ensures existence and uniqueness of solutions of the associated continuity equation \eqref{eq:conteq-uncontrolled}. Given an initial condition $\mu_0$, we denote with $\mu_t$ the corresponding (uncontrolled) solution and with \( \Phi^V \) the flow of the corresponding \emph{nonautonomous} vector field \( (t,x)\mapsto V_t(x,\mu_{t}) \).

We now set the geometric condition on $\omega$, the set on which the control is localized.
\begin{tcolorbox}
\textbf{Assumption \((A_2)\):} Given the initial datum of the reference trajectory $\varrho_0$, for any \( x\in \spt (\varrho_0) \) there exists \( t\in (0,T) \) such that \( \Phi_{0,t}(x)\in \omega \).
\end{tcolorbox}
This assumption is illustrated in Figure~\ref{f-geometric}.
As stated above, this condition is very natural, since it requires the reference trajectory to cross the set in which the localized control actually operates.
We highlight that this condition is a property of the reference trajectory and not of its perturbations.

\begin{figure}[htb]
  \begin{center}
    \includegraphics[width=0.55\textwidth]{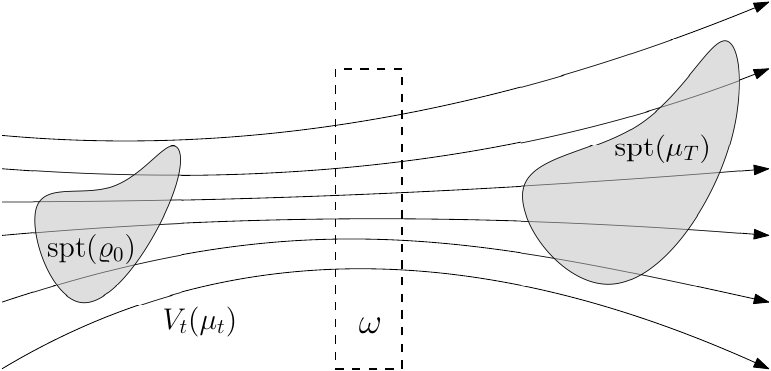}
\caption{Geometric Condition \( (A_2) \):
any trajectory of \( V_t(\mu_t) \) issuing from \( \spt(\varrho_0) \) crosses \( \omega \) by the time \( T \).}
\label{f-geometric}
\end{center}
\end{figure}

We are now ready to state the main result of the paper.

\begin{theorem}
  \label{thm:main}
  Let a nonlocal vector field \( V \), an open set \( \omega\subset \mathbb{R}^d \) and a measure \( \varrho_0\in \mathcal{P}_c(\mathbb{R}^d)\) satisfy Assumptions \( (A_{1,2}) \).
  Then, there exist $\kappa_{*},\varepsilon_{*}>0$ such that the set \( \Gamma_{\varepsilon_{*}}(\varrho_0) \) is \( \kappa \)-stabilizable around the reference trajectory $\mu_t$ of equation~\eqref{eq:conteq}, for any \( \kappa\in [0,\kappa_{*}) \).

  The numbers \( \kappa_{*} \), \( \varepsilon_{*} \) as well as \( C \) from the definition of \( \kappa \)-stabilizability only depend on \( T,M,\omega,\spt(\varrho_0)\).
\end{theorem}


A careful look at the structure of the stabilized set shows the main limitation of this result, since $\Gamma_{\varepsilon}(\varrho_0)$ is a very small subset of the Wasserstein ball $\overline{\bf B}_{\varepsilon}(\varrho_0)$.
Indeed, if $\varrho_0$ is composed of $N$ atoms, then $\Gamma_{\varepsilon}(\varrho_0)$ is finite-dimensional, since it always can be considered as a subset of $\R^{Nd}$.
If \( \varrho_0 \) is absolutely continuous, then \( \Gamma_\varepsilon(\varrho_0) \) is infinite dimensional, but still small: we will prove in Proposition \ref{p-nondense} below that \( \Gamma_\varepsilon(\varrho_0) \) is a compact, nowhere dense subset of \( \mathcal{P}_c^{ac}(\mathbb{R}^{d})\cap\overline{\bf B}_{\varepsilon}(\varrho_0)  \).

For this reason, we improve the main theorem by building a small neighborhood around $\Gamma_{\varepsilon}(\varrho_0)$ with respect to the Wasserstein distance.

\begin{corollary}
  \label{cor:main}
  Let all the assumptions of Theorem~\ref{thm:main} hold.
  Let \( \kappa_{*} \) and \( \varepsilon_{*} \) be as in Theorem~\ref{thm:main} and \( r(\varepsilon,\kappa) = -\varepsilon^{1+\kappa}/\log \varepsilon \).
  Then, for any \( \varepsilon\in (0,\varepsilon_{*}) \) and \( \kappa\in(0,\kappa_{*}) \), each point $\varrho\in{\bf B}_{r(\varepsilon,\kappa)}(\Gamma_{\varepsilon}(\varrho_0))$ can be steered by an admissible control into the ball \( {\bf B}_{C_1\varepsilon^{1+\kappa}}(\mu_T) \), where the constant \( C_1 \)
  only depends on \( T \), \( M \), \( \omega \), \( \spt(\varrho_0) \).
\end{corollary}



\begin{figure}[htb]
  \begin{center}
    \includegraphics[width=0.4\textwidth]{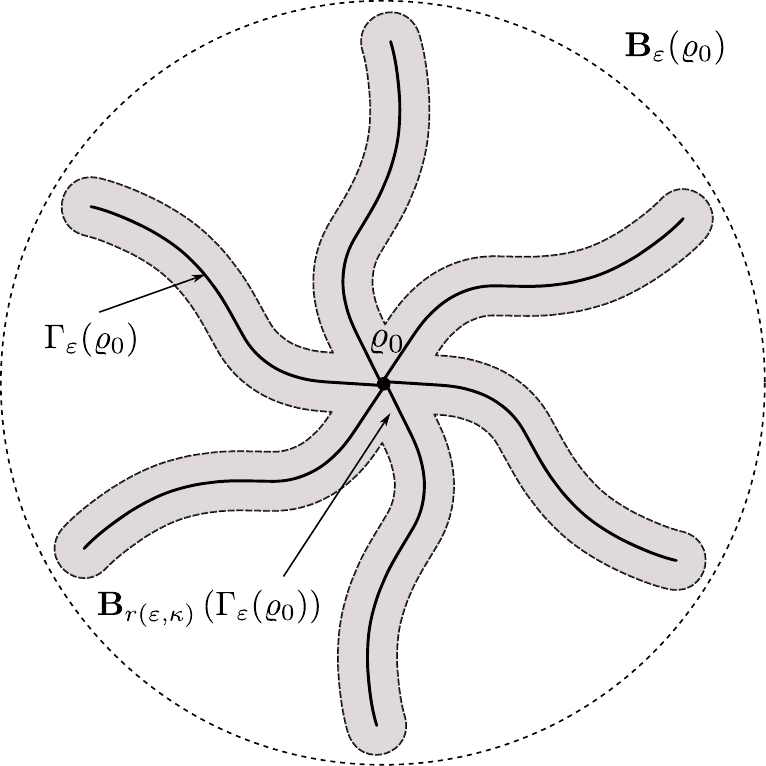}
    \caption{The set \( \Gamma_{\varepsilon}(\varrho_0) \) (black curves) and its enlargement \( {\bf B}_{r(\varepsilon,\kappa)}\left(\Gamma_{\varepsilon}(\varrho_0)\right) \) (grey region) inside \( {\bf B}_{\varepsilon}(\varrho_0) \).}
\label{fig:tree}
\end{center}
\end{figure}


One of the by-products of our result is that it provides some information about controllability (in reversed time). Indeed, consider the dynamics with reversed vector field $-V$ and localized Lipschitz controls: the reachable set starting from a very small neighborhood of \( \mu_T \) contains the set \(  \Gamma_{\varepsilon_{*}}(\varrho_0) \). We plan to explore this interesting possibility in future researches.

We now briefly discuss the rate of convergence.  We do not know if the estimate \( O(\varepsilon^{1+\kappa}) \) can be improved to \( \kappa>\kappa_{*} \).
The control \( u^{\varepsilon} \) constructed in the proof does not guarantee better results. This improvement might be of interest for future researches.


The paper is organized as follows:
Section~\ref{sec:aux} collects notations, as well as some known but useful results, including basic properties of nonlocal flows and bi-Lipschitz homeomorphisms.
The necessary information regarding the geometric structure of the Wasserstein space is provided in Section~\ref{sec:wass}.
Finally, Section~\ref{sec:proof} contains the proof of Theorem~\ref{thm:main} and a discussion about its sharpness.

\section{Ordinary differential equations and bi-Lipschitz homeomorphisms}
\label{sec:aux}
In this section, we first fix the notation. We then recall some basic properties of Ordinary Differential Equations (ODEs from now on) and give a criterion for checking that a map \( f\colon \mathbb{R}^d\to \mathbb{R}^d \) is a bi-Lipschitz homeomorphism.

\subsection{Notation}
We introduce some basic notation:

\begin{longtable}{p{.1\textwidth} p{.9\textwidth}}
  \( \langle x,y\rangle\) &\(\sum_{i=1}^dx_iy_i \), the Euclidean inner product on \( \mathbb{R}^{d} \)\\
  \( |x|\) &\(\sqrt{\langle x,x\rangle}  \), the Euclidean norm on \( \mathbb{R}^{d} \)\\
  \(\id\) & the identity map \(\id\colon \mathbb{R}^d\to \mathbb{R}^d\);\\
  \( \mathcal{L}(\mathbb{R}^d;\mathbb{R}^d) \) & the space of \( d \times d \) matrices or linear maps \( A\colon \mathbb{R}^d\to \mathbb{R}^d \)\\
  \( C(X;Y) \) & the space of continuous maps \( f\colon X\to Y \) between metric spaces \( X \), \( Y \)\\
  \( \Lip(X;Y) \) & the space of bounded Lipschitz maps \( f\colon X\to Y \)\\
  \( L^p(I;X) \) & the Lebesgue space of integrable maps (equivalence classes) \( f\colon I\to X \), \( p \geq 1 \)\\
  \( \langle A,B\rangle\) & \(\tr(A^{\T}B)= \sum_{i,j=1}^da_{ij}b_{ij}   \), the Frobenius inner product on \( \mathcal{L}(\mathbb{R}^{d};\mathbb{R}^{d}) \)\\
  \( \|A\| \) &\(\sqrt{\langle A,B\rangle}   \), the Frobenius norm on \( \mathcal{L}(\mathbb{R}^{d};\mathbb{R}^{d}) \)\\
	\( \|f\|_{\infty}\)&\(\esssup_{x \in \mathbb{R}^{d} }|f(x)|  \), the \( \sup \)-norm on \( L^{\infty}(\mathbb{R}^d;\mathbb{R}^d) \) or \( C(\mathbb{R}^d;\mathbb{R}^d) \)\\
  \( \lip (f) \) & \( \sup_{x \ne y }\frac{|f(x)-f(y)|}{|x-y|}  \), the minimal Lipschitz constant of \( f\colon \mathbb{R}^d\to \mathbb{R}^d \)\\
 \( \|f\|_{\lip} \) & \( \max\left\{\|f\|_{\infty},\lip(f)\right\}  \), the norm on \( \lip(\mathbb{R}^d;\mathbb{R}^d) \)\\
  \( \overline{B}_{r}(A)\) &\(\left\{x\;\colon\; \inf_{y\in A}|x-y|\le r\right\}\), the closed ball of radius \( r>0 \) around \( A \subset \mathbb{R}^d\)\\
  \( {B}_{r}(A)\) &\(\left\{x\;\colon\; \inf_{y\in A}|x-y|<r\right\} \), the corresponding open ball\\
  \(\co (A)\) & the convex hull of \(A\subset \mathbb R^d\)\\
  \(\mathscr L^d\) & the Lebesgue measure on \(\mathbb R^d\)\\
  \( \spt (\mu) \) & the support of the probability measure \( \mu \)\\
  \( \sharp \) & the pushforward operator (see Section~\ref{sec:wass})\\
  \(\mathcal{P}_{2}(\mathbb{R}^{d}) \) & the space of probability measures \( \mu \) on
                                         \( \mathbb{R}^{d} \) with \( \int |x|^{2}\,d\mu<\infty \)\\
  \( \mathcal{P}_{c}(\mathbb{R}^{d}) \) & the space of compactly supported probability measures on \( \mathbb{R}^{d} \)\\
  \( \mathcal{W}_2 \) & the 2-Wasserstein distance on \( \mathcal{P}_{2}(\mathbb{R}^{d}) \) (see Section~\ref{sec:wass})\\
  \( \Phi_{s,t} \) & the flow of a \emph{nonautonomous} vector field \( v\in L^1\left([0,T];\Lip(\mathbb{R}^d;\mathbb{R}^d)\right) \) on \( [0,T]^{2} \)\\
  \( O\left(g(\varepsilon)\right) \) & a \emph{non-negative} function \( f \) such that \( \limsup_{\varepsilon\to 0}\frac{f(\varepsilon)}{g(\varepsilon)}<\infty \), where \( g\ge0 \)\\
\end{longtable}

Remark that, in contrast to the standard Bachmann–Landau notation~\cite[p. 443]{grahamConcreteMathematicsFoundation1994}, we always assume that \( O(g(\varepsilon)) \) is non-negative.

\subsection{Matrix ordinary differential equations and differential inequalities}

Consider the following matrix differential equation:
\[
\dot X(t)=A(t)X(t) + B(t),\quad X(t)=M,
\]
where \( A,B\colon [0,T]\to \mathcal{L}(\mathbb{R}^d;\mathbb{R}^d) \) are measurable and bounded maps.
By differentiation, one can check that the solution is given by the formula
\begin{equation}
  \label{eq:odesol}
  X(t) = U(t) \left(M + \int_0^tU(s)^{-1}B(s)\,ds\right),
\end{equation}
where \( U(t) \) is the fundamental matrix, i.e., the unique solution of
\[
  \dot U(t)=A(t)U(t),\quad U(0)=\id.
\]

Formula~\eqref{eq:odesol} will be used later, together with the following lemma.

\begin{lemma}[Generalized Gr\"onwall inequality]
  \label{lem:diffineq}
  Let \( x=x(t) \) and \( y=y(t) \) satisfy
  \begin{gather*}
    \dot x(t) = a(t)x(t)+b(t),\quad x(0)=x_0,\\
    \dot y(t) \le a(t)y(t)+b(t),\quad y(0)=x_0,
  \end{gather*}
  where \( a,b\colon[0,T]\to \mathbb{R} \) are bounded measurable maps, \( x_0\in \mathbb{R} \).
  Then
  \[
  y(t)\le x(t)  := e^{\int_0^ta(s)ds}\left(x_0+ \int_0^te^{-\int_0^sa(\tau)d\tau}b(s)\,ds\right).
  \]
\end{lemma}

Lemma~\ref{lem:diffineq} is a direct consequence of~\eqref{eq:odesol} and~\cite[Theorem 4.1 p. 26]{hartmanOrdinaryDifferentialEquations2002}.
Remark that, in contrast to the standard Gr\"onwall's inequality, this estimate holds even if \( a \) and \( b \) are non-positive.

\subsection{Bi-Lipschitz homeomorphisms}
\label{sec:inv}
Here we develop a criterion allowing us to check whether a Lipschitz map is a bi-Lipschitz homeomorphism.
We begin with definitions.

\begin{definition}
We say that a map \( f\colon \mathbb{R}^d\to \mathbb{R}^d \) is compactly supported if it equals the identity outside a compact set.
The smallest compact set with this property will be denoted by \( \spt (f) \).
\end{definition}

\begin{definition}
A map \( f\colon \mathbb{R}^d\to \mathbb{R}^d \) is called \emph{bi-Lipschitz} if there exists \( L>0 \) such that
\[
L^{-1}|x-y|\le |f(x)-f(y)|\le L|x-y|,
\]
for all \( x,y\in \mathbb{R}^d \).
\end{definition}

\begin{definition}
Let \( f\colon \mathbb{R}^d\to \mathbb{R}^d \) be locally Lipschitz. \emph{Clarke's generalized Jacobian}~\cite{clarkeInverseFunctionTheorem1976, clarkeFunctionalAnalysisCalculus2013} is defined as follows:
\[
\partial_{C} f(x):= \co\left\{\lim_{i\to\infty}Df(x_i)\;\colon\; x_i\to x,\; x_{i}\not\in E\cup E_f \right\},
\]
where \( E\subset \mathbb{R}^d \) is any set of measure zero, \( E_{f} \) is the set of points where \( f \) is non-differentiable.
\end{definition}

We have the following result.

\begin{theorem}
  \label{thm:inv}
	A compactly supported Lipschitz map \( f\colon \mathbb{R}^d\to \mathbb{R}^d \) is a bi-Lipschitz homeomorphism if all elements of \( \partial_C f \) are invertible.
\end{theorem}

\begin{proof}
  It follows from Clarke's inverse function theorem~\cite{clarkeInverseFunctionTheorem1976} that \( f \) is a local bi-Lipschitz homeomorphism.
  In particular, for each \( x \) there exists an open set \( \mathcal{O}_x \), containing \( x \), and such that \( f\colon \mathcal{O}_x\to f(\mathcal{O}_x) \) is a homeomorphism.
  Notice that \( \spt (f) \) can be covered by a finite number of such sets, say \( (\mathcal{O}_{i})_{i=1}^n \).

  Fix some \( y\in \spt (f) \) and consider the set
  \[
	I:=\left\{ i \;\colon\; y\in f(\mathcal{O}_{i})\right\}.
  \]
  Now, \( U =\bigcap_{i\in I} f(\mathcal{O}_{i}) \) is an open neighborhood of \( y \).
  Therefore each \( \mathcal{O}_{i} \), \( i\in I\), contains a set \( U_{i} \) such that \( f\colon U_i\to U \) is a homeomorphism.
  Clearly, for any \( i,j\in I\) there are only two possibilities: either \( U_i=U_j \) or \( U_i\cap U_j=\varnothing\).
  Thus \( f^{-1}(U) \) is a finite union of open disjoint sets homeomorphic to \( U \).
  This proves that \( f \) is a covering map.
  But for a covering map the number of sheets should be the same for all \( y \); hence it must be \( 1 \), because there are points where \( f =\id\).
  In other words, \( f \) is a homeomorphism.
  Its bi-Lipschitz continuity is another consequence of Clarke's inverse function theorem.
\end{proof}

\subsection{Classical flows and regular perturbations}
\label{subsec:reg_pert}
In this section, we introduce a notion of regular perturbation and discuss its relation with the starting set \( \Gamma_{\varepsilon_0}(\varrho_0) \).

Note that \( \Gamma_{\varepsilon_0}(\varrho_0) \) contains very precise information about the way one perturbs \( \varrho_0 \).
Indeed, it is composed of measures obtained by perturbing \( \varrho_0 \) with (classical) flows of Lipschitz vector fields.
In fact, one may consider more general perturbations than flows.
Such perturbations are described in the following definition.




\begin{definition}[Regular perturbations]
  \label{def:reg_pert}
Fix a measure \( \varrho_0\in \mathcal{P}_c(\mathbb{R}^d) \) and a pair of numbers \( \varepsilon_0>0 \) and \( L>0 \).
  We say that a curve \( \varrho\colon [0,1]\to \mathcal{P}_c(\mathbb{R}^{d}) \) belongs to the class \( \Pi_{\varepsilon_0,L}(\varrho_0) \) of \emph{regular perturbations} of \( \varrho_0 \) if there exists a map \( \Psi\colon [0,\varepsilon_0] \times \mathbb{R}^d\to \mathbb{R}^d \)
such that
\begin{enumerate}[(a)]
  \item for all \( \varepsilon\in [0,\varepsilon_0] \), \( \Psi_{\varepsilon} \) is \( L \)-Lipschitz:
        \begin{equation}
          \label{eq:Psi1}
        \left|\Psi_{\varepsilon}(x)-\Psi_{\varepsilon}(y)\right|\le L|x-y|\quad \forall x,y\in \mathbb{R}^d;
        \end{equation}
  \item for all \( \varepsilon\in [0,\varepsilon_0] \), \( \Psi_{\varepsilon} \) is ``close to the identity'':
        \begin{equation}
          \label{eq:Psi2}
         \left|\Psi_{\varepsilon}(x)-x\right| \le \varepsilon L\quad \forall x\in \mathbb{R}^d;
        \end{equation}
  \item for all \( \varepsilon\in [0,\varepsilon_0] \) and \( x\in \mathbb{R}^d \), the matrix \( D\Psi_{\varepsilon}(x) \) is positive semi-definite as soon as it exists.
  \item one has \( \Psi_{\varepsilon\sharp}\varrho_{\varepsilon}=\varrho_0 \), for all \( \varepsilon\in [0,\varepsilon_0] \).
\end{enumerate}
\end{definition}
While the definition seems technical, it simply lists minimal properties of a perturbation that we require to prove the main theorem.

With a class of admissible perturbation at hand, we can define the set of initial conditions:
\begin{equation}
  \label{eq:Gamma2}
  \Gamma_{\varepsilon_{*}}^{\varepsilon_0,L}(\varrho_0):= \left\{\varrho_{\varepsilon}\text{ where } \varrho\in \Pi_{\varepsilon_0,L}(\varrho_0),\,\varepsilon\in[0,\varepsilon_{*}]\right\},\quad \varepsilon_{*}\in [0,\varepsilon_0].
\end{equation}
The relation between \( \Gamma_{\varepsilon_*}(\varrho_0) \), defined by~\eqref{eq:Gamma}, and \( \Gamma^{\varepsilon_0,L}_{\varepsilon_{*}}(\varrho_0) \) is established in the following lemma.

\begin{lemma}
  \label{lem:flows}
  One has \( \Gamma_{\varepsilon_*}(\varrho_0) \subset \Gamma^{\log 2,2}_{\varepsilon_{*}}(\varrho_0) \), for all \( \varepsilon_{*}\in [0,\log 2] \).
\end{lemma}
\begin{proof}
  Let \( \mu\in \Gamma_{\varepsilon_*}(\varrho_0) \).
   Then there exists a vector field \( v \in L^\infty\left([0,\varepsilon_*];\Lip(\mathbb{R}^d;\mathbb{R}^d)\right) \) such that \( \|v\|_{\infty}\le 1 \) and the curve \( \varrho\colon [0,\varepsilon_*]\to \mathcal{P}_2(\mathbb{R}^d) \) defined by \( \varrho_{\varepsilon} := (\Phi_{0,\varepsilon})_{\sharp}\varrho_0   \) passes through \( \mu \).

  Recall that
  \[
    \Phi_{s,t}(x) = x + \int_s^t v_{\tau}\left(\Phi_{s,\tau}(x)\right)\,d\tau,\quad \forall s,t\in [0,\varepsilon_*].
  \]
  In particular, the inequality \( \|v\|_{\infty}\le 1 \) implies that
  \[
    \left|\Phi_{s,t}(x)-x\right|\le |s-t|,\quad
    \left|\Phi_{s,t}(x)-\Phi_{s,t}(x)\right|\le |x-y|e^{|s-t|},
  \]
  for all \( x,y\in \mathbb{R}^d \) and \( \varepsilon\in [0,\varepsilon_*] \).

 Take some \( x,h\in \mathbb{R}^d \), \( t\in \mathbb{R} \), and consider the identity
  \begin{equation}
    \label{eq:Phi_diff}
    \frac{\Phi_{\varepsilon,0}(x+th) - \Phi_{\varepsilon,0}(x)}{t} = h - \frac{1}{t}\int^\varepsilon_0\left[v_{\varepsilon-s}\left(\Phi_{\varepsilon,\varepsilon-s}(x+th)\right)-v_{\varepsilon-s}\left(\Phi_{\varepsilon,\varepsilon-s}(x)\right)\right]\,ds.
  \end{equation}
  We estimate the second term on the right-hand side:
  \[ \left|\frac{1}{t}\int_0^{\varepsilon}\left[v_{\varepsilon-s}\left(\Phi_{\varepsilon,\varepsilon-s}(x+th)\right)-v_{\varepsilon-s}\left(\Phi_{\varepsilon,\varepsilon-s}(x)\right)\right]\,ds\right|\le |h|\int_0^{\varepsilon}e^{ s}\,ds = |h|(e^{\varepsilon}-1).
  \]
  If \( D\Phi_{\varepsilon,0}(x) \) exists, by passing to the limits as \( t\to 0 \) in~\eqref{eq:Phi_diff}, it holds
  \[
    D\Phi_{\varepsilon,0}(x)h = h + M(\varepsilon,x)h,
  \]
  for some matrix \( M(\varepsilon,x) \).
  We know that \( \|M(\varepsilon,x)\|\le e^{\varepsilon}-1  \), for any \( \varepsilon\in [0,\varepsilon_*] \) and \( x\in \mathbb{R}^d \).
  Thus, if \( \varepsilon_*\le\log 2 \), then \( \|M(\varepsilon,x)\|\le 1 \) and \( D\Phi_{\varepsilon,0}(x) \) is positive semi-definite.

  Finally, let \( \Psi_{\varepsilon}:=\Phi_{\varepsilon,0} \) and note that~\eqref{eq:Psi1} and~\eqref{eq:Psi2} hold for \( L= 2 \) whenever \( \varepsilon\in [0,\log 2] \), while \( D\Psi_{\varepsilon}(x) \) is positive semi-definite when exists.
  This observation completes the proof.
\end{proof}

We conclude this section by studying the topology of $\Gamma_r(\rho_0)$ for $\rho_0\in  \mathcal{P}^{ac}_c(\R^d)$. As stated above, we will show that it is much smaller than the Wasserstein ball $\overline{\bf B}_r(\rho_0)$ and, in particular, it is not dense. For more general results, see e.g. \cite{gigliInverseImplicationBrenierMccann2011}.
\begin{proposition}
  \label{p-nondense}
  Let $\rho_0\in \mathcal{P}^{ac}_c(\R^d)$ be fixed.
  Then, $\Gamma_\varepsilon(\rho_0)$ defined by \eqref{eq:Gamma} is a compact nowhere dense subset of \( \mathcal{P}_c^{ac}(\mathbb{R}^d)\cap \overline{\bf B}_r(\varrho_0) \).
\end{proposition}
\begin{proof}
  We first prove that $\Gamma_\varepsilon(\rho_0)$ is compact in \( \mathcal{P}_2(\mathbb{R}^d) \).
  First observe that the set
  $$
  V:=\{v\in C(\mathbb{R}^{d};\mathbb{R}^d)
\text{ with } \|v\|_{\lip}\le 1\}
  $$
  is closed in \( C(\mathbb{R}^d;\mathbb{R}^d) \) if this space is endowed with the topology of \emph{local uniform convergence}.
  Note that \( \Gamma_\varepsilon(\varrho_0) \) is the reachable set at time \( r \) of the differential inclusion
  $$
  \partial_t \mu_{t} \in -\nabla\cdot(V\mu_t),\quad \mu_0=\varrho_0,
  $$
  that should be understood in the sense of~\cite{bonnet-weillViabilityInvarianceProper2023}.
  By~\cite[Proposition 2.21]{bonnet-weillViabilityInvarianceProper2023}, \( \Gamma_\varepsilon(\varrho_0) \) is a compact subset of \( \mathcal{P}_2(\mathbb{R}^d) \).
  Since the Wasserstein topology and the weak topology coincide on compact sets, we conclude that \( \Gamma_\varepsilon(\varrho_0) \) is compact in the weak topology as well.

  We now show that \( \Gamma_\varepsilon(\varrho_0)\subset \mathcal{P}_c^{ac}(\mathbb{R}^d)\cap \overline{\bf B}_r(\varrho_0) \).
On the one hand, \(\Gamma_\varepsilon(\varrho_0)\) is completely contained in $\mathcal{P}^{ac}_c(\R^d)$, by regularity of flows of the continuity equation with Lipschitz vector fields starting from $\rho_0\in \mathcal{P}^{ac}_c(\R^d)$.
On the other hand, \( \Gamma_\varepsilon(\varrho_0)\subset \overline{\bf B}_r(\varrho_0) \) by Proposition~\ref{prop:basic} Statement 5 and the second inequality in~\eqref{eq:W2ineq} below.


It remains to show that \( \Gamma_r(\varrho_0) \) has empty interior in $ \mathcal{P}_c^{ac}(\mathbb{R}^d)\cap \overline{\bf B}_r(\varrho_0)$, and therefore, being a compact set, is nowhere dense in \( \mathcal{P}_c^{ac}(\mathbb{R}^d)\cap \overline{\bf B}_r(\varrho_0) \).
Suppose the contrary, i.e., there exists an open subset \( \mathcal{O} \) of \( \mathcal{P}_c(\mathbb{R}^d)\cap \overline{\bf B}_r(\varrho_0) \) such that \( \mathcal{P}_c^{ac}(\mathbb{R}^d)\cap \mathcal{O}\subset \Gamma_r(\varrho_0) \).
It is known that \( \mathcal{O} \) contains an atomic measure \( \mu_d \).
Since any atomic measure can be approximated in the weak topology by an absolutely continuous measure with uniformly compact support, there exists a sequence \( \mu_{k}\in \mathcal{P}_c^{ac}(\mathbb{R}^d)\cap \mathcal{O}\) converging to \( \mu_d \).
Now, by compactness, it holds \( \mu_d\in \Gamma_r(\varrho_0) \), which is contrary to the fact that \( \Gamma_r(\varrho_0)\subset \mathcal{P}_c^{ac}(\mathbb{R}^d) \).
\end{proof}

\subsection{Nonlocal flows}
In this section, we discuss basic properties of the nonlocal continuity equation~\eqref{eq:conteq-uncontrolled} under the regularity assumption \( (A_1) \).

\begin{proposition}
  \label{prop:basic}
  Let Assumption \( (A_1) \) hold.
  Then,
\begin{enumerate}[\rm (1)]
	\item For each \( \varrho_0\in \mathcal{P}_{c}(\mathbb{R}^d) \) there exists a unique solution \( Z^{\varrho_0}\colon [0,T]\times \mathbb{R}^d \to \mathbb{R}^d \) of the Cauchy problem on \( C(\mathbb{R}^d;\mathbb{R}^d) \):
        \begin{equation}
          \label{eq:flowZ}
	\dot Z^{\varrho_0}_t = V_t(Z^{\varrho_0}_t,Z^{\varrho_0}_{t\sharp}\varrho_0),\quad Z^{\varrho_0}_0=\id.
        \end{equation}
	\item For each \( \varrho_0\in \mathcal{P}_{c}(\mathbb{R}^d) \) the curve \( \mu_t := Z^{\varrho_0}_{t\sharp}\varrho_0 \) is the unique weak solution of the nonlocal continuity equation~\eqref{eq:conteq-uncontrolled} with the initial condition \( \mu_0=\varrho_0 \).
	\item The map \( Z^{\varrho_0}_t\colon \mathbb{R}^d\to \mathbb{R}^{d} \), called \emph{the flow of \( V \)}, is a bi-Lipschitz homeomorphism for each \( t\in[0,T] \):
\[
	e^{-Mt}|x-y| \le \left|Z^{\varrho_0}_t(x)-Z^{\varrho_0}_t(y)\right| \le e^{Mt}|x-y|.
				\]
  \item If \( V \) is continuously differentiable in \( x \), then, for each \( t\in [0,T] \), the map \( Z_t^{\varrho_0}\colon \mathbb{R}^d\to \mathbb{R}^d \) is a \( C^1 \)-diffeomorphism.
  \item For all \( t\in [0,T] \) and \( x\in \mathbb{R}^d \), one has
        \[
        \left|Z_t^{\varrho_0}(x)-x\right|\le Mt.
        \]
\end{enumerate}
\end{proposition}

Proofs can be found, e.g., in~\cite{piccoliTransportEquationNonlocal2013}.

Slightly abusing the notation, we will refer to \( Z_t^{\varrho_0} \) as the \emph{nonlocal flow}.
When \( \varrho_0 \) is fixed and no confusion is possible, we will write \( Z_t \) instead of \( Z^{\varrho_0}_t \) to simplify the notation.

Remark that~\eqref{eq:flowZ} is a particular case of the ``multi-agent system in the Lagrangian form'':
\begin{equation}
  \label{eq:Lagrange}
    \dot X^{\varrho_0}_t(x) = F_t\left(x,X^{\varrho_0}_t(x), X^{\varrho_0}_{t\sharp}\varrho_0\right),\quad X_0^{\varrho_0}(x)=x.
  \end{equation}
Such equations were considered, for instance, in~\cite{cavagnariLEK}.

\begin{proposition}
  \label{prop:Lagrange}
  Let \( F=F_t(x,y,\mu) \) be measurable in \( t \), \( M \)-Lipschitz in \( x \), \( y \), \(\mu\), for some \( M>0 \), and bounded.
  Then,~\eqref{eq:Lagrange} has a unique solution \( X^{\varrho_0} \).
  Moreover, the map \( X^{\varrho_0}_t \) is \( (1+MT)e^{Mt} \)-Lipschitz for each \( t\in [0,T] \).
\end{proposition}
\begin{proof}
  The existence and uniqueness of \( X^{\varrho_0} \) follows from~\cite[Proposition 4.8]{cavagnariLEK}.
  We now show that \( X_t^{\varrho_0} \) is Lipschitz for all \( t\in [0,T] \).
  Since
  \[
    X^{\varrho_0}_t(x) = x + \int_0^tF_s\left(x,X^{\varrho_0}_s(x), X^{\varrho_0}_{s\sharp}\varrho_0\right)\,ds,
  \]
  we conclude that
  \begin{align*}
    \left|X^{\varrho_0}_t(x)-X^{\varrho_0}_t(y)\right|
    &\le |x-y| + M\int_0^t \left(|x-y| + \left|X^{\varrho_0}_s(x)-X^{\varrho_0}_s(y)\right|\right)\,d s\\
    &\le (1+MT)|x-y| + M\int_0^t \left|X^{\varrho_0}_s(x)-X^{\varrho_0}_s(y)\right|\,d s,
  \end{align*}
  for all \( x,y\in \mathbb{R}^d \) and \( t\in [0,T] \).
  By Gr\"onwall's lemma, we have that \( X_t^{\varrho_0} \) is Lipschitz, with Lipschitz constant \( (1+MT)e^{Mt} \), for each \( t\in [0,T] \).
\end{proof}

In contrast to~\eqref{eq:flowZ}, the solution of~\eqref{eq:Lagrange} is not necessarily a homeomorphism for each \( t \).
As an example, if \( F_t(x,y,\mu) = -x \), then \( X_t^{\varrho_0}(x) = x -tx \).
Hence, for \( t=1 \), we have \( X_{t}^{\varrho_0}\equiv 0 \).

\section{Wasserstein space}
\label{sec:wass}


In this section, we collect several definitions and standard results about the structure of  \emph{Wasserstein spaces}.
For a more detailed discussion, we refer to the books \cite{AGS05,santambrogioOptimalTransportApplied2015,villaniTopicsOptimalTransportation2003,zbMATH05306371}.

We begin by reviewing the standard concepts of pushforward measure and transport plan.

\begin{definition}
  Let \(\mu\) be a Borel probability measure on \(\mathbb R^m\) and  \( f\colon \mathbb{R}^m\to \mathbb{R}^n \) be a Borel map.
  The probability measure \(f_\sharp\mu\) on \(\mathbb R^n\) defined by
\begin{displaymath}
  (f_{\sharp}\mu)(A):= \mu\left(f^{-1}(A)\right),\quad \text{for all Borel sets }A\subset Y,
\end{displaymath}
is called \emph{the pushforward} of \(\mu\) by \(f\).
\end{definition}

\begin{definition}
Let \(\mu\) and  \(\nu\) be any probability measures on \(\mathbb R^d\).
A \emph{transport plan} \( \Pi \) between \(\mu\) and \(\nu\) is a probability measure on \( \mathbb{R}^{d}\times \mathbb{R}^{d} \) whose projections on the first and the second factor are
\( \mu \) and \( \nu \), respectively. In other words, 
\begin{displaymath}
  \pi^{1}_{\sharp}\Pi = \mu,\quad \pi^{2}_{\sharp}\Pi=\nu,
\end{displaymath}
where \( \pi^{1}, \pi^{2} \) are the projection maps on the factors, i.e., \( (x,y)\mapsto x \) and \( (x,y)\mapsto y \), respectfully.
The set of all transport plans  between \( \mu \) and \( \nu \) is denoted by  \( \Gamma(\mu,\nu) \).
\end{definition}


\begin{definition}
The space \( \mathcal{P}_{2}(\mathbb{R}^{d}) \) of probability measures \(\mu\) on \( \mathbb{R}^d \) with finite second moments, i.e., such that \(\int|x|^2\,d\mu(x)<\infty\) is called the quadratic \emph{Wasserstein space}.
\end{definition}

\begin{definition}
  Given a pair of measures \( \mu,\nu\in \mathcal{P}_2(\mathbb{R}^d) \), the corresponding \emph{optimal transportation problem} is the minimization problem
  \[
  \displaystyle \inf_{\Pi\in\Gamma(\mu,\nu)}\int|x-y|^{2}\,d \Pi(x,y).
  \]
  Any transport plan \( \Pi\in \Gamma(\mu,\nu) \) that solves it is called \emph{optimal}.
 \end{definition}




One can use optimal transportation to define a  distance on the space \( \mathcal{P}_{2}(\mathbb{R}^{d}) \), turning it into a metric space.

\begin{definition} The Wasserstein distance is
\begin{equation}
  \label{eq:wasserstein}
  \mathcal{W}_2(\mu,\nu) := \left(\inf_{\Pi\in\Gamma(\mu,\nu)}\int|x-y|^{2}\,d \Pi(x,y)\right)^{1/2}.
\end{equation}
\end{definition}
\begin{proposition} The Wasserstein distance is a distance~ on  \( \mathcal P_{2}(\mathbb{R}^{d}) \).
The space \( \mathcal P_{2}(\mathbb{R}^{d}) \) equipped with this distance is a \emph{complete separable metric space}.
\end{proposition}
\begin{proof}
    See \cite[Theorem 7.3]{villaniTopicsOptimalTransportation2003} and \cite[Theorem~6.18]{zbMATH05306371}.
\end{proof}

\begin{proposition} For any Lipschitz maps \( f,g\colon \mathbb{R}^d\to \mathbb{R}^d \), it holds
\begin{equation}
  \label{eq:W2ineq}
  \mathcal{W}_2(f_{\sharp}\mu,f_{\sharp}\nu)\le \lip(f) \mathcal{W}_2(\mu,\nu),\quad \mathcal{W}_2(f_{\sharp}\mu,g_{\sharp}\mu)\le
  \left(\int |f-g|^2\,d\mu\right)^{1/2},
\end{equation}
\end{proposition}
\begin{proof}
    See e.g. \cite{piccoliTransportEquationNonlocal2013}.
\end{proof}




\section{Proof of the main theorem}
\label{sec:proof}
In this section, we prove Theorem~\ref{thm:main}.
We first define in Section~\ref{subsec:cutoff} a specific cut-off function, which we use later to write a candidate stabilizing control.
The construction of this control is discussed in Sections~\ref{subsec:local_stable}, \ref{subsec:nonlocal}, first for the linear and then for the general case.
In these sections, we also show that the obtained control is well-defined and admissible, which is the most technical part of the proof.
In Section~\ref{subsec:stabil} we prove that the chosen control is indeed stabilizing.
Throughout Sections~\ref{subsec:local_stable}--\ref{subsec:stabil} we impose an additional regularity assumption \( (A_3) \) on \( V \); we show in Section~\ref{subsec:A3} that it can be eliminated.
Section~\ref{subsec:corollary} contains a proof of Corollary~\ref{cor:main}.
In Section \ref{s-sharp}, we finally discuss sharpness of the result.

The theorem will be first proved with the following additional assumption:
\begin{tcolorbox}
  \textbf{Temporary assumption \((A_3)\):} \( V \) is continuously differentiable in \( x \).
\end{tcolorbox}

Later, in Section~\ref{subsec:A3}, we will show that this assumption can be omitted.

In fact, we will prove a slightly more general version of Theorem~\ref{thm:main}, that deals with sets \( \Gamma^{\varepsilon_0,L}_{\varepsilon_{*}}(\varrho_0) \), defined by~\eqref{eq:Gamma2}, instead of \( \Gamma_{\varepsilon_*}(\varrho_0) \).
It can be stated as follows.

\begin{theorem}
  \label{thm:main2}
  Let the assumptions of Theorem~\ref{thm:main} hold and \( L>0 \) be fixed.
  Then, there exist $\kappa_{*}>0$ and \( \varepsilon_{*}\in (0,\varepsilon_0) \) such that the set \( \Gamma^{\varepsilon_{0},L}_{\varepsilon_{*}}(\varrho_0) \) is \( \kappa \)-stabilizable around the reference trajectory $\mu_t$ of equation~\eqref{eq:conteq}, for any \( \kappa\in (0,\kappa_{*}) \).

  The numbers \( \kappa_{*} \), \( \varepsilon_{*} \) as well as \( C \) from the definition of \( \kappa \)-stabilizability only depend on \( \varepsilon_0,L,T,M,\omega,\spt(\varrho_0)\).
\end{theorem}

\subsection{Cut-off function}
\label{subsec:cutoff}

In this section, we define a cut-off function, that will be used later to choose a stabilizing control.
We first study some properties of the set \( \omega \).

\begin{lemma}
	\label{lem:common_time}
	Let Assumptions \( (A_{1,2}) \) hold. Then, there exist \( r>0 \) and a nonempty compact set \( \omega_0 \) with the following properties:
	\begin{itemize}
		\item \( \overline{B}_{r}(\omega_0)\subset \omega \);
		\item
		      for any \( x\in \spt (\varrho_0) \) the ball \( \overline{B}_{r}\left(Z_{t}(x)\right) \) belongs to \( \omega_0\) for a set of times \( t \) containing an interval of length \( r/M\).
	\end{itemize}
\end{lemma}
\begin{proof}
  Let \( K:=\spt (\varrho_0) \), that is compact.
	By Assumption \( (A_{2}) \), for any \( x\in K \), there exists \( t(x)>0 \) such that \( Z_{t(x)}(x)\in\omega \).
	Since \( \omega \) is open, we can find a \emph{closed} ball \( \overline{B}_{\frac{7}{2}R(x)}\left(Z_{t(x)}(x)\right)\subset \omega \) with \( R(x)>0 \).
	Since \( Z_t \) is a homeomorphism there is an \emph{open} ball \( B_{\tilde R(x)}(x) \) with \( \tilde R(x)>0 \) such that
	\[
		Z_{t(x)}\left(B_{\tilde R(x)}(x)\right)\subset \overline{B}_{R(x)}\left(Z_{t(x)}(x)\right)\subset \omega.
	\]
	By compactness of \( K \), there exists a finite set of points \( x_i\in K \), \( 1 \le i \le N \), such that
	\[
		K\subset \bigcup_{i=1}^{N}B_{\tilde{R}(x_{i})}(x_i).
	\]
	By boundedness of \( v \), it holds
	\[
		Z_{s}(x) \in \overline{B}_{M|s-t|}\left(Z_t(x)\right)\quad \forall x\in \mathbb{R}^{d}\quad \forall t,s >0.
	\]
	Then
	\[
		Z_{s}\left(B_{\tilde{R}(x_{i})}(x_{i})\right)\subset \overline{B}_{R(x_{i})+M|s-t(x_{i})|}\left(Z_{t(x_{i})}(x_{i})\right)\quad \forall s\in \mathbb{R}^{+}.
	\]
	That also implies
	\[
		\overline{B}_{R(x_{i})}\left(Z_{s}\left(B_{\tilde{R}(x_{i})}(x_{i})\right)\right)\subset \overline{B}_{2R(x_{i})+M|s-t(x_{i})|}\left(Z_{t(x_{i})}(x_{i})\right)\quad \forall s\in \mathbb{R}^{+}.
	\]
	Hence, if
	\[
		2R(x_{i})+M|s-t(x_{i})|\le \frac{5}{2}R(x_{i}),
	\]
	i.e.,
	\[
		s\in \left(t(x_i)-\frac{R(x_{i})}{2M}, t(x_i)+\frac{R(x_{i})}{2M}\right),
	\]
	the ball \( \overline{B}_{R(x_{i})}\left(Z_{s}(x)\right) \) belongs to \( \overline{B}_{\frac{5}{2}R(x_i)}\left(Z_{t(x_{i})}(x_{i})\right) \) for all \( x\in B_{\tilde{R}(x_i)}(x_i) \).
	We choose
	\[
		\omega_0 := \bigcup_{i=1}^{N}\overline{B}_{\frac{5}{2}R(x_i)}\left(Z_{t(x_i)}(x_i)\right),\quad r := \min\left\{R(x_{i})\;\colon\;1 \le i \le N\right\},
	\]
	to complete the proof.
\end{proof}

We now define a useful cut-off function \( \alpha\colon \mathbb{R}^d\to \mathbb{R} \).
First, we define an auxiliary function \( \alpha_0 \colon \omega_{0} \cup \left(\mathbb{R}^d\setminus  B_{r}(\omega_{0}) \right)\to \mathbb{R} \) by the rule
\[
	\alpha_{0}(x) =
	\begin{cases}
		0, & x\in \mathbb{R}^d\setminus B_{r}(\omega_{0}), \\
		a, & x\in \omega_0,
	\end{cases}
\]
where \( a \) is a \emph{positive} number to be defined later. Then, we extend it to the whole space \( \mathbb{R}^d \) as follows:
\begin{equation}
  \label{eq:alpha}
	\alpha(x) := \inf_{y\in\omega_{0} \cup \big(\mathbb{R}^d\setminus B_r(\omega_{0}) \big) }\left\{\alpha_{0}(y) + \frac{a}{r}|x-y|\right\}.
\end{equation}
The resulting map \( \alpha\colon \mathbb{R}^d\to \mathbb{R} \) has Lipschitz constant \( \lip(\alpha) = a/r \), and its range belongs to the segment \( [0,a] \).

\begin{remark}
  \label{rem:r}
Below, we will always assume that the cut-off function \( \alpha \) is \emph{continuously differentiable}.
If this is not the case,  we can always replace \(\alpha\) with its ``mollified'' version \( \alpha^{\varepsilon} = \eta^{\varepsilon}*\alpha \), where \( \eta^{\varepsilon} \) is the standard mollification kernel:
\begin{equation}
\label{eq:mollifier}
\eta^{\varepsilon}(x) = \varepsilon^{-d}\phi(x/\varepsilon),
\quad
\phi(|x|) =
\begin{cases}
\beta e^{-\frac{1}{1-|x|^2}} &|x|\le1\\
0 &|x|>1
\end{cases},
\quad
\int \phi\,dx=1.
\end{equation}
After this change, the upper bound of the cut-off function and its Lipschitz constant do not increase.
On the other hand, mollification changes the support: \( \spt (\alpha^{\varepsilon}) \subset  \overline{B}_{\varepsilon}(\spt (\alpha)) \).
Still, after a suitable adjustment of \( r \), it holds \( \spt (\alpha^{\varepsilon})\subset \omega \) for a sufficiently small \( \varepsilon \).
\end{remark}

\subsection{Admissible control: linear case}
\label{subsec:local_stable}
Here we define a (potentially) stabilizing control for the case in which the vector field \( V_t \) does not depend on the measure, i.e., \( V_t(x,\mu)=V_t(x) \).





Let \( \kappa>0 \) and \( \varepsilon\in(0,1) \).
Recall that the cut-off function \( \alpha \) defined in~\eqref{eq:alpha} has the free parameter \(a\).
We choose \( a \) so that \(e^{-\frac{ar}{M}} = \varepsilon^{\kappa}\), i.e.,
\begin{equation}
\label{eq:a}
a:= - \frac{\kappa M}{r}\log \varepsilon,
\end{equation}
where \( r \) is the constant in Lemma~\ref{lem:common_time}, eventually adjusted with Remark~\ref{rem:r}.

Throughout this section, we consider a measure \( \varrho_0\in \mathcal{P}_c(\mathbb{R}^d) \) and a regular perturbation \( \varrho\in \Pi_{\varepsilon_0,L}(\varrho_0) \) with given \( \varepsilon_0>0 \) and \( L>0 \).
We also denote by \( \Psi\colon [0,\varepsilon_0] \times \mathbb{R}^d \to \mathbb{R}^d  \) the corresponding map from Definition~\ref{def:reg_pert}.

We now consider three Cauchy problems in \( C(\mathbb{R}^d;\mathbb{R}^d) \):
\begin{gather}
\label{eq:initial}
\dot Z_t = V_t(Z_t),\quad Z_0 = \id,\\
\label{eq:target}
\dot Y_t = V_t(Y_t),\quad Y_0 = \Psi_{\varepsilon},\\
\label{eq:controlled}
\dot X_t = V_t(X_t) + \alpha(X_t)(Y_t - X_t),\quad X_0 = \id.
\end{gather}

Equation \eqref{eq:initial} is associated with the evolution of \emph{the perturbed measure} \( \varrho_{\varepsilon} \) under the vector field \( V_t \).
Indeed, the pushforward measure \( \mu^{\varepsilon}_t:=Z_{t\sharp}\varrho_{\varepsilon} \) satisfies the continuity equation
\[
\partial_t\mu^{\varepsilon}_t+ \nabla\cdot(V_t\mu^{\varepsilon}_t)=0,\quad \mu^{\varepsilon}_0=\varrho_{\varepsilon}.
\]

Equation \eqref{eq:target} is associated with the evolution of \emph{the reference measure} \( \varrho_0 \) under \( V_t \).
Indeed, by Definition~\ref{def:reg_pert}(d), we have that \( \mu_t:=Y_{t\sharp}\varrho_{\varepsilon} = Z_{t\sharp}(\Psi_{\varepsilon\sharp}\varrho_{\varepsilon})=Z_{t\sharp}\varrho_0 \), i.e., \( \mu_t \) satisfies
\[
\partial_t\mu_t+ \nabla\cdot(V_t\mu_t)=0,\quad \mu_0=\varrho_0.
\]

Finally, Equation \eqref{eq:controlled} can be regarded as Equation \eqref{eq:target} perturbed by the control function \( \alpha(X_t)(Y_t-X_t) \).
A priori, we cannot say that \eqref{eq:controlled} is related to any continuity equation, since \( \alpha(X_t(x))(Y_t(x)-X_t(x)) \) depends not only on \( X_t(x) \) but also on \( x \).
On the other hand, if we prove that \( X_t \) is invertible and the inverse is Lipschitz, then
\begin{equation}
  \label{eq:control}
u_t(x) := \alpha(x)\left(Y_t\circ X^{-1}_t(x)-x\right)
\end{equation}
would be an admissible control and \( \mu^{u}_t:= X_{t\sharp}\varrho_{\varepsilon} \) would satisfy
\[
\partial_t\mu^{u}_t+ \nabla\cdot\left((V_t+u_t)\mu^{u}_t\right)=0,\quad \mu^u_0=\varrho_{\varepsilon}.
\]

The following proof of Theorem~\ref{thm:main2} consists of two steps.
In the first step, we show that \( u \) given by~\eqref{eq:control} is indeed well-defined and admissible for all sufficiently small \( \varepsilon>0 \).
Since \( u_t \) is defined with the aim of  ``pushing'' towards the reference trajectory while passing through \( \omega \), we will be able to prove, in the second step, that \( \mu_t^{u} \) lies inside a tube of order \( \varepsilon^{1+\kappa} \) around the reference path  \( \mu_t \) for all \( t \) which are close enough to \( T \).
This situation is depicted in Figure~\ref{fig:paths}.


\begin{remark}
  Let us stress that the control \( u_t \) defined by~\eqref{eq:control} depends on \( \varepsilon \), since \( Y \) depends on this variable through the initial condition \( Y_0=\Psi_{\varepsilon} \).
  Below, we do not write explicitly this dependence in order to simplify the notation.
\end{remark}

From now on, we assume that \( (A_{1\text{-}3}) \) hold.
We begin by discussing the differentiability of the maps \( Z_t \), \( Y_t \), \( X_t \), \( t\in [0,T] \).
The first one, \( Z_t \), is continuously differentiable due to Assumptions \( (A_1) \), \( (A_3) \) and Proposition~\ref{prop:basic} Statement 4.
The second one, \( Y_t \), is differentiable at \( x \) when \( x \) is a differentiability point of \( \Psi_{\varepsilon} \).
This follows from the representation \( Y_t = Z_t\circ \Psi_{\varepsilon} \).
The corresponding derivative equals \( DZ_t(\Psi_{\varepsilon}(x))D\Psi_{\varepsilon}(x) \).
The differentiability of \( X_t \) is a more delicate issue, that we solve in the next lemma.
Before passing to the lemma, note that \( X_t \) is Lipschitz for all \( t\in [0,T] \), due to Proposition~\ref{prop:Lagrange}, and thus differentiable for almost all \( x\in \mathbb{R}^d \).

\begin{lemma}
  \label{lem:Xdiff}
  Let Assumptions \( (A_{1\text{-}3}) \) hold.
  Fix a number \( \varepsilon\in (0,\varepsilon_0) \) and a point \( x\in \dom(D \Psi_{\varepsilon}) \).
  Then, \( X_t \) is differentiable at \( x \) and \( DX_t(x)= R_t \), where \( R_t\in \mathcal{L}(\mathbb{R}^d;\mathbb{R}^d) \) is the unique solution of the matrix differential equation
  \[
    \begin{cases}
     \frac{d}{dt}R_t = \left(DV_t(X_t(x)) - \alpha(X_t(x))\id-\left(X_t(x)-Y_t(x)\right)D\alpha(X_t(x))\right)R_t + \alpha(X_t(x))DY_t(x),\\ R_0=\id.
    \end{cases}
  \]
\end{lemma}
\begin{proof}
  Throughout the proof, the point \( x \in \dom(D \Psi_{\varepsilon}) \) will be fixed.

Equation~\eqref{eq:controlled} can be rewritten as follows:
\[
\dot X_t(x) = W_t(X_t(x))+\alpha(X_t(x))Z_t(\Psi_{\varepsilon}(x)),\quad X_0(x)=x,
\]
Here \( W_t(y) := V_t(y) - \alpha(y)y \).
This vector field is clearly bounded, continuously differentiable in \( x \) and uniformly Lipschitz, i.e., there exists \( L_{W} >0\) such that
\[
\left|W_t(y_1)-W_t(y_2)\right|\le L_{W}|y_1-y_2|,
\]
for all \( t\in[0,T] \) and all \( y_1,y_2\in \mathbb{R}^d \).

The corresponding linearized equation reads as follows
\[
\dot R_t = DW_t(X_t(x))R_t+Z_t(\Psi_{\varepsilon}(x))D\alpha(X_t(x))R_t + \alpha(X_t(x))DZ_t(\Psi_{\varepsilon}(x))D \Psi_{\varepsilon}(x),\quad R_0=\id.
\]
This equation is well-defined because \( x\in \dom(D\Psi_{\varepsilon}) \).
We aim to show that \( X_t \) is differentiable at \( x \) for all \( t\in[0,T] \), with derivative \( DX_t(x) \) coinciding with \( R_t \).

Consider the map \( \mathcal{F}\colon [-1,1] \times C([0,T];\mathbb{R}^d)\to C([0,T];\mathbb{R}^d) \) defined by
\[
\mathcal{F}(\lambda,\phi)_{t} := x+\lambda h +\int_0^t \left(W_t(\phi_s+ \alpha(\phi_{s}Z_{s}\circ \Psi_{\varepsilon}(x+\lambda h)\right)\,ds,\quad x\in \mathbb{R}^d,
\]
where \( h \) is an arbitrary unit vector in \( \mathbb{R}^d \). The space \( C([0,T];\mathbb{R}^d) \) is equipped with the norm\footnote{Remark that this norm is equivalent to the usual supremum norm.}
\[
\|\phi\|_{\sigma} := \max_{t\in [0,T]}\left\{e^{-\sigma t}|\phi_{t}|\right\},
\]
where \( \sigma >0\) will be chosen later.

The continuity of \( \lambda\mapsto \mathcal{F}(\lambda,\phi) \) is obvious.
Now, we will show that \( \phi\mapsto \mathcal{F}(\lambda,\phi) \) is contractive for some \( \sigma \).
We deduce from~\eqref{eq:Psi1} that \( |\Psi_{\varepsilon}(x+\lambda h)|\le |\Psi_{\varepsilon}(x)|+L\lambda h \).
Hence, thanks to Proposition~\ref{prop:basic} Statement 5, there exists \( C>0 \) such that \( \left|Z_s\circ \Psi_{\varepsilon}(x+\lambda h)\right|\le C \) for all \( \lambda\in[-1,1] \) and \( h\in B_1(0) \).
Therefore,
\begin{align*}
  \left|\mathcal{F}(\lambda,\phi)_t - \mathcal{F}(\lambda,\psi)_t\right|
  &\le \int_0^s\left(\left|W_s(\phi_s)-W_s(\psi_s)\right| + \left|\alpha(\phi_s)-\alpha(\psi_s)\right|\cdot \left|Z_s\circ \Psi_{\varepsilon}(x+\lambda h)\right|\right)\,ds\\
  &\le \int_0^s\left(L_{W}+C\lip(\alpha)\right)\left|\phi_s-\psi_s\right|\,ds\\
  &\le \int_0^se^{\sigma t}\left(L_{W}+C\lip(\alpha)\right)\,ds\cdot\left\|\phi-\psi\right\|_{\sigma}\\
  &\le e^{\sigma t}\left(L_{W}+C\lip(\alpha)\right)\sigma^{-1}\left\|\phi-\psi\right\|_{\sigma},
\end{align*}
which implies
\[
  \left\|\mathcal{F}(\lambda,\phi) - \mathcal{F}(\lambda,\psi)\right\|_{\sigma}\le \eta \left\|\phi-\psi\right\|_{\sigma},
\]
where \( \eta = \left(L_{W}+C\lip(\alpha)\right)\sigma^{-1} \).
We choose \( \sigma := 2\left(L_{W}+C\lip(\alpha)\right)\), thus guaranteeing the contraction property with \( \eta=1/2 \).

We now use the parametrized version of Banach's contraction mapping theorem~\cite[Theorem A.2.1]{zbMATH05197323}.
According to this theorem, \( \mathcal{F}(\lambda,\cdot) \) admits a unique fixed point \( \phi^{\lambda} \):
\[
\phi^{\lambda}_{t} = x+\lambda h +\int_0^t \left(W_t(\phi^{\lambda}_s)+ \alpha(\phi^{\lambda}_{s})Z_{s}\circ \Psi_{\varepsilon}(x+\lambda h)\right)\,ds.
\]
Moreover, this fixed point satisfies the inequality
\begin{equation}
  \label{eq:sigmaineq}
\|\phi^0 + \lambda Rh - \phi^{\lambda} \|_{\sigma}\le 2\|\phi^0 + \lambda Rh - \mathcal{F}(\lambda,\phi^0+\lambda Rh)\|_{\sigma}.
\end{equation}

By construction, it holds \( \phi^{\lambda}_t = X_t(x+\lambda h) \).
Therefore, the left-hand side of~\eqref{eq:sigmaineq} is
\[
\|X(x+\lambda h) - X(x)-\lambda Rh\|_{\sigma},
\]
while
\begin{multline}
  \label{eq:fun_rhs}
   \mathcal{F}(\lambda,\phi^0+\lambda Rh)_t-\phi^0_t - \lambda R_th=\\
  x+\lambda h + \int_0^t\left(W_s\left(X_s(x)+\lambda R_sh\right)+\alpha\left(X_s(x)+\lambda R_sh\right)Z_s\circ \Psi_{\varepsilon}(x+\lambda h)\right)\,ds \\
  -
  \left(x+\int_0^t\left(W_s(X_s(x))+\alpha(X_s(x))Z_s\circ \Psi_{\varepsilon}(x)\right)\,ds\right)- \lambda R_th.
\end{multline}
Since \( \alpha \), \( W_s \), \( Z_s \) are differentiable and \( \Psi_{\varepsilon} \) is differentiable at \( x \), we can write
\begin{equation}
  \label{eq:os}
  \begin{cases}
  W_s\left(X_s(x)+\lambda R_sh\right) = W_s(X_s(x)) + \lambda DW_s(X_s(x)) R_sh+o_s(\lambda),\\
  \alpha\left(X_s(x)+\lambda R_sh\right) = \alpha(X_s(x)) + \lambda D\alpha(X_s(x)) R_sh+o_s(\lambda),\\
    Z_s\circ \Psi_{\varepsilon}(x+\lambda h) = Z_s\circ\Psi_{\varepsilon}(x) + \lambda DZ_s(\Psi_{\varepsilon}(x))D\Psi_{\varepsilon}(x)h + o_s(\lambda).
  \end{cases}
\end{equation}
Here we use \( o_s(\lambda) \) as a shorthand for a function which is measurable in \( s \), continuous in \( \lambda \), bounded in the sense that
\begin{equation}
  \label{eq:obound}
|o_s(\lambda)|\le C_1 \quad \forall s\in [0,T]\quad\forall \lambda\in [-1,1],
\end{equation}
for some \( C_1>0 \), and satisfies
\begin{equation}
  \label{eq:osmall}
\lim_{\lambda\to 0}\frac{o_s(\lambda)}{\lambda}=0\quad \forall s\in [0,T].
\end{equation}

By using~\eqref{eq:os}, we can restate the right-hand side of~\eqref{eq:fun_rhs} as follows:
\begin{align*}
  \lambda h+\lambda\int_0^tDW_s(X_s(x))R_s\,ds\cdot h +  \lambda \int_0^t\alpha(X_s(x))DZ_s(\Psi_{\varepsilon}(x))D\Psi_{\varepsilon}(x)\,ds \cdot h \\
  + \lambda \int_0^t Z_s(\Psi_{\varepsilon}(x)) D\alpha(X_s(x))R_s\,ds\cdot h + \int_0^to_s(\lambda)\,ds - \lambda R_th.
\end{align*}
Since
\[
R_t = \id + \int_0^t\left(DW_s(X_s(x))R_s+Z_s(\Psi_{\varepsilon}(x))D\alpha(X_s(x))R_s + \alpha(X_s(x))DZ_s(\Psi_{\varepsilon}(x))D \Psi_{\varepsilon}(x)\right)\,ds,
\]
the inequality~\eqref{eq:sigmaineq} becomes
\[
\|X(x+\lambda h) - X(x)-\lambda Rh\|_{\sigma}\le 2\max_{t\in[0,T]}\Big\{e^{-\sigma t}\Big|\int_0^to_s(\lambda)\,ds\Big|\Big\}.
\]
Now, from~\eqref{eq:obound}, \eqref{eq:osmall} and Lebesgue's dominated convergence theorem it follows that
\[
\lim_{\lambda\to0}\frac{1}{\lambda}\|X(x+\lambda h) - X(x)-\lambda Rh\|_{\sigma} = 0.
\]
For completing the proof, it remains to note that the formula holds for any arbitrary unit vector \( h \).
\end{proof}

\begin{lemma}
  \label{lem:Xinv}
  Let Assumptions \( (A_{1\text{-}3}) \) hold and \( \kappa \) in~\eqref{eq:a} satisfy \( \kappa\in\left(0,\frac{r}{TM}\right) \).
  Then, there exists \( \varepsilon_1\in (0,\varepsilon_0) \) such that, for each \( \varepsilon\in (0,\varepsilon_1) \), the solution \( X_t \) of~\eqref{eq:controlled} is a bi-Lipschitz homeomorphism for each \( t\in [0,T] \).
The constant \( \varepsilon_1 \) above depends on \( M \), \( T \), \( r \), \( \kappa \), \( L \), \( \varepsilon_0 \) only.
\end{lemma}
\begin{proof}

 The proof contains several steps.

\vspace{5pt}
\noindent\emph{Step 1: Show that \( |X_t-Y_t|\le O(\varepsilon) \) and \( |X_t-Z_t|\le O(\varepsilon) \), for all \( t\in [0,T] \), with \( O(\varepsilon) \) determined by \( M \), \( T \), \( L \) only.}
\vspace{5pt}

We have
\[
\frac{d}{dt}(X_t-Y_t) = V_t(X_t)-V_t(Y_t) - \alpha(X_t)(X_t-Y_t).
\]
Therefore,
\[
\frac{d}{dt}|X_t-Y_t|^{2} \le 2M|X_t-Y_t|^{2} - 2\alpha(X_t)|X_t-Y_t|^{2},
\]
which means that
\[
|X_t-Y_t|^2\le |\id - \Psi_{\varepsilon}|^{2}e^{2Mt}\cdot e^{-2\int_0^t\alpha(X_s)ds}
\]
and, thanks to~\eqref{eq:Psi2}, we obtain
\begin{equation}
|X_t-Y_t|\le \varepsilon L e^{Mt}\cdot e^{-\int_0^t\alpha(X_s)ds}\le \varepsilon Le^{MT}.
\label{eq:X-Y}
\end{equation}


We now use the Lipschitz continuity of \( Z_t \) proved in Proposition~\ref{prop:basic} Statement 3.
We also write \( Y_t=Z_t\circ \Psi_{\varepsilon} \).
By~\eqref{eq:Psi2}, we obtain that \(|Y_t-Z_t| \le \varepsilon Le^{MT} \), for all \( t\in[0,T] \).
Therefore, \(|X_t-Z_t| \le 2\varepsilon Le^{MT}\), for all \( t\in[0,T] \).

\vspace{5pt}
\noindent\emph{Step 2: Prove that \( \|DY_t(x)\|\le Le^{Mt} \) and \( \|DZ_t(x)\|\le e^{Mt} \) for all \(x\in \dom(D \Psi_{\varepsilon})\).}
\vspace{5pt}

Recall that \( X_t \) and \( Y_t \) are differentiable for all \( x\in \dom(D\Psi_{\varepsilon})\).
We deduce from
\[
\frac{d}{dt} DY_t(x) = DV_t(X_t(x))DY_t(x),\quad DY_0(x) = D \Psi_{\varepsilon}(x)
\]
and~\eqref{eq:Psi1} that
\[
  \|DY_t(x)\|\le \|D \Psi_{\varepsilon}(x)\|\cdot e^{Mt}\le Le^{Mt}\quad \forall x\in\dom(D \Psi_{\varepsilon}).
\]
To deal with \( \|DZ_t(x)\| \), it suffices to start from \( \id \) instead of \( \Psi_{\varepsilon} \).

\vspace{5pt}
\noindent\emph{Step 3: Derive estimates for auxiliary matrix equations on the set \( \dom(D \Psi_{\varepsilon}) \).}
\vspace{5pt}

These are equations of the form:
\begin{gather*}
  \frac{d}{dt}U_t = DV_t(X_t)U_t,\quad U_0=\id,\\
\frac{d}{dt}Q_t = \left(DV_t(X_t) -(X_t-Y_t)D\alpha(X_t)\right)Q_t,\quad Q_0=\id,
\end{gather*}
considered for all \( x\in\dom(D \Psi_{\varepsilon}) \).
Remark that \( DY_t = U_tD \Psi_{\varepsilon}  \).

First, we estimate \( \|U_t-Q_t\| \).
It follows from \( |X_t-Y_t| \le O(\varepsilon)\), \( |D\alpha(X_t)| \le O(-\log\varepsilon)\) and
\begin{align*}
  \frac{d}{dt}(U_t-Q_t) = \left(DV_t(X_t)-(X_t-Y_t)D\alpha(X_t)\right)(U_t-Q_t) +(X_t-Y_t)D\alpha(X_t)U_t
\end{align*}
that
\begin{align*}
  \frac{1}{2}\frac{d}{dt}\|U_t-Q_t\|^2 \le \left(M +O(-\varepsilon\log\varepsilon)\right)\|U_t-Q_t\|^2 +O(-\varepsilon\log\varepsilon)  \|U_t-Q_t\|.
\end{align*}
Therefore, it holds
\begin{align*}
  \|U_t-Q_t\|^2
  &\le O(-\varepsilon\log\varepsilon)\int_0^t\|U_s-Q_s\|e^{2\int_s^t\left(M+O(-\varepsilon\log\varepsilon)\right)d\tau}\,ds\\
  &\le O(-\varepsilon\log\varepsilon)\sup_{s \in [0,t]} \|U_s-Q_s\| \int_0^te^{2\int_s^t\left(M+O(-\varepsilon\log\varepsilon)\right)d\tau}\,ds.
\end{align*}
By taking supremum over all \( t\in[0,T] \) on the both sides, we obtain
\begin{equation}
  \label{eq:U-Q}
\sup_{t \in [0,T]} \|U_t-Q_t\|\le O(-\varepsilon\log\varepsilon).
\end{equation}
Finally, from \( \left\|DV_t(X_t)+(X_t-Y_t)D\alpha(X_t)\right\|\le M+O(-\varepsilon\log\varepsilon) \) it follows that
\begin{equation}
  \label{eq:invQ}
\sup_{t \in[0,T] } \|Q_t^{-1}\| \le O(1)\quad\text{for all}\quad \varepsilon\in[0,\varepsilon_0].
\end{equation}

By recalling~\eqref{eq:a} and Step 1, we have that \( O(-\varepsilon\log\varepsilon) \) in the right-hand side of~\eqref{eq:U-Q} and \( O(1) \) in the right-hand side of~\eqref{eq:invQ} are determined by \( M \), \( T \), \( \kappa \), \( r \), \( L \) only.

\vspace{5pt}
\noindent\emph{Step 4: Show that there exists \( \varepsilon_1\in (0,\varepsilon_0) \) such that, for any \( \varepsilon\in (0,\varepsilon_1) \), the matrix \( DX_t(x) \) is invertible for any \( x\in \dom(D\Psi_{\varepsilon})\).}
\vspace{5pt}

All functions in this step are restricted to \( \dom(D\Psi_{\varepsilon})\).
It follows from Lemma~\ref{lem:Xdiff} that
\[
\frac{d}{dt}DX_t = \left(DV_t(X_t) - \alpha(X_t)\id- (X_t-Y_t)D\alpha(X_t)\right)DX_t +\alpha(X_t)DY_t.
\]
Therefore,
\begin{align}
  DX_t &= e^{-\int_0^t\alpha(X_s)ds}Q_t\left(\id+\int_0^t\alpha(X_s)e^{\int_0^s\alpha(X_\tau)d\tau}Q_s^{-1}DY_s\,ds\right)\notag\\
 &= Q_t\left(e^{-\int_0^t\alpha(X_s)ds}\id+\int_0^t\alpha(X_s)e^{-\int_s^t\alpha(X_\tau)d\tau}Q_s^{-1}U_s\,ds\cdot D \Psi_{\varepsilon}\right).\label{eq:DX}
\end{align}

Define
\[
S_t = \int_0^t\alpha(X_s)e^{-\int_s^t\alpha(X_\tau)d\tau}Q_s^{-1}U_s\,ds - \int_0^t\alpha(X_s)e^{-\int_s^t\alpha(X_\tau)d\tau}Q_s^{-1}Q_s\,ds,
\]
that satisfies
\begin{align}
  \|S_t\|&\le  \int_0^t\alpha(X_s)e^{-\int_s^t\alpha(X_\tau)d\tau}\left\|Q_s^{-1}U_s-Q_s^{-1}Q_s\right\|\,ds\notag\\
  &\le
  \int_0^t\alpha(X_s)e^{-\int_s^t\alpha(X_\tau)d\tau}\left\|Q_s^{-1}\right\|\cdot\left\|U_s-Q_s\right\|\,ds= O(\varepsilon\log^{2}\varepsilon).
         \label{eq:S_t_est}
\end{align}
Here we used~\eqref{eq:U-Q}, \eqref{eq:invQ} and the fact that
\begin{equation}
  \label{eq:intalpha}
\int_0^t\alpha(X_s)e^{-\int_s^t\alpha(X_\tau)d\tau}\le O(-\log\varepsilon).
\end{equation}

Now, \eqref{eq:DX} yields \( DX_t=Q_tP_t \) with
\begin{align}
  \label{eq:DX_t}
  P_t =e^{-\int_0^t\alpha(X_s)ds}\id+D\Psi_{\varepsilon}\int_0^t\alpha(X_s)e^{-\int_s^t\alpha(X_\tau)d\tau}\,ds+S_t\cdot D \Psi_{\varepsilon}.
\end{align}
The matrix \( Q_t(x) \) is obviously invertible for all \( x\in \dom (D\Psi_{\varepsilon})\).
We claim that the matrix \( P_t(x) \) is positive definite for all \( x\in \dom(D\Psi_{\varepsilon}) \), when \( \varepsilon \) is small.
Since, by Definition~\ref{def:reg_pert}, Statement (c), \( D\Psi_{\varepsilon} \) is positive semi-definite for all \( \varepsilon \in (0,\varepsilon_0) \), we conclude that
\[
  \langle P_tv,v\rangle \ge e^{-\int_0^t\alpha(X_s)ds}\|v\|^2 + \left\langle S_t\cdot (D \Psi_{\varepsilon})v,v\right\rangle
  \ge \left(e^{-\int_0^t\alpha(X_s)ds} - \|S_t\|\cdot \|D\Psi_{\varepsilon}\|\right)\|v\|^2,
\]
for all \( v\in \mathbb{R}^d \) and \( \varepsilon\in (0,\varepsilon_0) \).
It remains to note that
\[
e^{-\int_0^t\alpha(X_s)\,ds}\ge e^{-\int_0^T\alpha(X_s)\,ds} = e^{-Ta} = \varepsilon^{\frac{\kappa MT}{r}},
\]
then use~\eqref{eq:S_t_est} and \( \|D\Psi_{\varepsilon}\| \le L \), to get
\[
  \langle P_tv,v\rangle
  \ge \varepsilon^{\frac{\kappa TM}{r}}\left(1 - O(\varepsilon^{1-\frac{\kappa TM}{r}}\log^2\varepsilon)\right)\|v\|^2.
\]
Therefore, if \( \kappa <\frac{r}{TM} \), then there exist \( \varepsilon_{1},c>0 \) such that
  \begin{equation}
    \label{eq:Ppos}
  \langle P_t(x)v,v\rangle
  \ge \varepsilon^{\frac{\kappa TM}{r}}c\|v\|^2\quad \forall \varepsilon\in (0,\varepsilon_{1})\quad \forall x\in \dom(D\Psi_{\varepsilon})\quad \forall v\in \mathbb{R}^{d}.
  \end{equation}
By analyzing the above estimates, we conclude that \( \varepsilon_1 \) and \( c \) are determined by \( M \), \( T \), \( r \), \(\kappa\), \( L \), \( \varepsilon_0 \) only.
Now, for each \( \varepsilon\in (0,\varepsilon_{1}) \), \( DX_t \) is invertible on \( \dom(D\Psi_{\varepsilon})\cap\overline{B}_{\delta}(0) \) as the product of invertible matrices \( Q_t \) and \( P_t \).

\vspace{5pt}
\noindent\emph{Step 5: Show that \( X_t \) is bi-Lipschitz for all \( \varepsilon\in (0,\varepsilon_1) \) and \( t\in [0,T] \).}
\vspace{5pt}

Fix any \( \varepsilon\in (0,\varepsilon_1)  \). We have already proved the result for $x\in \dom(D\Psi_{\varepsilon})$. We now prove the result for $x\in  \mathbb{R}^d \setminus \dom(D\Psi_{\varepsilon})$, that has zero Lebesgue measure due to the fact that \( \Psi_{\varepsilon} \) is Lipschitz. We prove bi-Lipschitz continuity by studying the Clarke's generalized Jacobian \( \partial_CX_t(x) \) (see Section~\ref{sec:inv}).

Let $x\in  \mathbb{R}^d \setminus \dom(D\Psi_{\varepsilon})$ and a sequence \( x_k\to x \) such that \( x_k\in \dom(D\Psi_{\varepsilon}) \) be given. For each \( x_{k} \), we have \( DX_t(x_k)=Q_t(x_k)P_t(x_k) \)
and for \( P_t(x_k) \) the inequality~\eqref{eq:Ppos} holds.
Since the matrix \( Q_t(y) \) is invertible for each \( y\in \mathbb{R}^d \) and continuous as a function of \( y \), we conclude that \( DX_t(x_k) \) converges if and only if \( P_t(x_k) \) converges, and in this case
\[
\lim_{k\to\infty}DX_t(x_k) = Q_t(x)\lim_{k\to\infty}P_t(x_k).
\]

 Recall now that matrices $DX_t(y)$ are uniformly bounded for $y\in\R^d$, since their norm is bounded by the Lipschitz constant of $X_t$ given in Proposition \ref{prop:Lagrange}. We then consider a subsequence $x_{k_l}$ such that both $DX_t(x_{k_l})$ and $P_t(x_{k_l})$ are converging. We denote with $ P^{\{x_{k_l}\}}_t(x)$ the limit of the matrix $P_t$, highlighting that it depends on the chosen subsequence. Since \( c>0 \) from Step 4 does not depend on \( x \), the limiting matrix satisfies the inequality
\begin{equation}
    \label{e-Pinv}
\langle P^{\{x_{k_l}\}}_t(x)v,v\rangle\ge \varepsilon^{\frac{\kappa TM}{r}}c|v|^2 \mbox{~~ for all~~}  v\in \mathbb{R}^{d} .
\end{equation}
Since this estimate does not depend on the chosen subsequence, it holds for all limits of the bounded sequence $P_t(x_k)$. Given another sequence $y_k\to x$ with $y_k\in \dom(D\Psi_{\varepsilon})$, the same estimate \eqref{e-Pinv} holds too. Finally, the estimate holds for any convex combination \( (1-\lambda)P^{\{x_k\}}_t(x)+\lambda P^{\{y_k\}}_t(x) \) of such matrices, i.e. for all matrices in the Clarke's generalized Jacobian \( \partial_CX_t(x) \) (see Section~\ref{sec:inv}). Since \eqref{e-Pinv} ensures that matrices are invertible, we have proved that \( \partial_CX_t(x) \) consists exclusively of invertible matrices. Now, we use Theorem~\ref{thm:inv}, to complete the proof\footnote{In fact, Theorem~\ref{thm:inv} requires \( X_t \) to be compactly supported, i.e. \( X_t=\id \) outside a compact set \( K \). Since all our measures are compactly supported and \( V_t \) is bounded, we can assume, without loss of generality, that \( V_t=0 \) outside some compact set \( K \). Then Theorem~\ref{thm:inv} can be formally applied.}.
\end{proof}

\subsection{Admissible control: nonlocal case}
\label{subsec:nonlocal}

We now adapt the results of Section~\ref{subsec:local_stable} to the nonlocal case. It is worth noting that neither formula~\eqref{eq:control} nor  the statements and the proofs of Lemmas~\ref{lem:Xdiff} and \ref{lem:Xinv} are affected by the appearance of the nonlocal term (up to the trivial replacement of \( DV(X_t) \) with \( DV(X_t,X_{t\sharp}\varrho_{\varepsilon})\)), except for the first step in the proof of Lemma~\ref{lem:Xinv}.
Below, we provide a detailed proof for this step in the nonlocal case.


\begin{proofof}{Step 1 in Lemma~\ref{lem:Xinv} (nonlocal case)}
As before, we consider the Cauchy problems
\begin{gather*}
  \dot Z_t = V_t(Z_t,Z_{t\sharp}\varrho_{\varepsilon}),\quad Z_0 = \id,\\
  \dot Y_t = V_t(Y_t,Y_{t\sharp}\varrho_{\varepsilon}),\quad Y_0 = \Psi_{\varepsilon},\\
  \dot X_t = V_t(X_t,X_{t\sharp}\varrho_{\varepsilon}) + \alpha(X_t)(Y_t - X_t),\quad X_0 = \id.
\end{gather*}

  Our goal is to show that \( |X_t-Y_t|\le O(\varepsilon) \) and \( |X_t-Z_t|\le O(\varepsilon) \), for all \( t\in [0,T] \), with \( O(\varepsilon) \) determined by \( M \), \( T \), \( L \), \( \varepsilon_0 \) only.

We have
\[
\frac{d}{dt}(X_t-Y_t) = V_t(X_t,X_{t\sharp}\varrho_{\varepsilon})-V_t(Y_t,Y_{t\sharp}\varrho_{\varepsilon}) - \alpha(X_t)(X_t-Y_t).
\]
Therefore,
\begin{align}
  \frac{d}{dt}|X_t-Y_t|^2
  &\le 2\left\langle V_t(X_t,X_{t\sharp}\varrho_{\varepsilon})-V_t(Y_t,Y_{t\sharp}\varrho_{\varepsilon}),X_t-Y_t\right\rangle - 2\alpha(X_t)|X_t-Y_t|^2\notag\\
  &\le 2M|X_t-Y_t|^{2} + 2M\mathcal{W}_2(X_{t\sharp}\varrho_{\varepsilon},Y_{t\sharp}\varrho_{\varepsilon})|X_t-Y_t|-2\alpha(X_t)|X_t-Y_t|^2\notag\\
  &\le 2M|X_t-Y_t|^{2} + 2M|X_t-Y_t|\Big(\!\int\! |X_t-Y_t|^2\,d\varrho_{\varepsilon}\!\Big)^{1/2}\!\!\!\!-2\alpha(X_t)|X_t-Y_t|^2.
    \label{eq:X-Ynl}
\end{align}
We integrate both sides with respect to \( \varrho_{\varepsilon} \) and use the Cauchy–Schwarz inequality
\begin{equation}
\label{eq:L1-L2}
\int|X_t-Y_t|\cdot 1\,d\varrho_{\varepsilon}\le \left(\int|X_t-Y_t|^2\,d\varrho_{\varepsilon}\right)^{1/2}\cdot \left(\int 1^2\,d\varrho_{\varepsilon}\right)^{1/2}.
\end{equation}
This gives
\[
\frac{d}{dt}\int|X_t-Y_t|^2\,d\varrho_{\varepsilon}\le(4M-2\alpha(X_t))\int|X_t-Y_t|^2\,d\varrho_{\varepsilon},
\]
and therefore
\begin{equation}
  \label{eq:X-Ynl2}
\int|X_t-Y_t|^2\,d\varrho_{\varepsilon}\le  e^{4Mt} e^{-2\int_0^t\alpha(X_s)ds} \int|\id-\Psi_{\varepsilon}|^2\,d\varrho_{\varepsilon}.
\end{equation}
We conclude from~\eqref{eq:X-Ynl2} and~\eqref{eq:Psi2} that \( \int|X_t-Y_t|^2\,d\varrho_{\varepsilon} \le O(\varepsilon^2) \).
Because of~\eqref{eq:L1-L2}, we have \( \int|X_t-Y_t|\,d\varrho_{\varepsilon} \le O(\varepsilon) \).

Now, we deduce from~\eqref{eq:X-Ynl} that
\[
\frac{d}{dt}|X_t-Y_t|^2
  \le 2M|X_t-Y_t|^{2} + |X_t-Y_t|O(\varepsilon),
\]
which gives
\[
  |X_t-Y_t|^2\le e^{2Mt}\left(|\id -\Psi_{\varepsilon}|^{2} + O(\varepsilon)\int_0^t|X_s-Y_s|e^{-2Ms}\,ds\right).
\]

Define \( \Delta := \sup_{t \in [0,T]}|X_t-Y_t| \).
We deduce from the previous inequality that
\[
  \Delta^2\le e^{2MT}\varepsilon^2 + O(\varepsilon) \Delta.
\]
This means that \( \Delta \) cannot exceed the greatest root of the quadratic equation \(y^2 + by +c = 0\),
where \( b = - O(\varepsilon)  \), \( c = -e^{2MT}\varepsilon^2 \).
That is, \( \Delta\le O(\varepsilon) \).
As in the linear case, \( O(\varepsilon) \) here is determined by \( M \), \( T \), \( L \), \( \varepsilon_0 \) only.

Finally, note that the inequality \( |X_t-Z_t|\le O(\varepsilon) \) can be proved as in the linear case.
\end{proofof}

\subsection{Stabilization property}
\label{subsec:stabil}
Here we prove that the control \( u \) defined by~\eqref{eq:control} is indeed stabilizing, that is, there exist \( C,\varepsilon_{*}>0 \) such that \( W_2(\varrho_0,\varrho_{\varepsilon})< \varepsilon \) implies
\begin{equation}
  \label{eq:epspower}
  \mathcal{W}_2\left(\mu_T^{u},\mu_T\right)=
  \mathcal{W}_2\left(X_{T\sharp}\varrho_{\varepsilon},Y_{T\sharp}\varrho_{\varepsilon}\right)< C\varepsilon^{1+\kappa},
\end{equation}
for all \( \kappa\in\left(0,\frac{r}{TM}\right) \) and \( \varepsilon\in (0,\varepsilon_{*}) \).
Note that, due to~\eqref{eq:Psi2}, the inequality  \( W_2(\varrho_0,\varrho_{\varepsilon})< \varepsilon \) holds a priori for all \( \varepsilon<1/L \).
Here and below we use the notation of Sections~\ref{subsec:local_stable}, \ref{subsec:nonlocal}.

By the inequality~\eqref{eq:Psi2} and the identity \( \varrho_0=\Psi_{\varepsilon\sharp}\varrho_{\varepsilon} \), it holds
\[
\spt (\varrho_{\varepsilon}) \subset \overline{B}_{L\varepsilon}(\spt(\varrho_{0}))\quad \forall \varepsilon\in [0,\varepsilon_0].
\]
Using this fact and the Lipschitz continuity of \( Z_t \), i.e.,
\[
|Z_t(x)-Z_t(x_0)|\le e^{MT}|x-x_0|\quad \forall t\in [0,T]\quad \forall x,x_0\in \mathbb{R}^d,
\]
we deduce that for all \( \varepsilon\in [0,\varepsilon_0] \) such that \( e^{MT} L\varepsilon<r/2 \) the following holds:
\begin{equation}
  \label{eq:middle}
\forall x\in \spt(\varrho_{\varepsilon})\; \exists\, x_0\in \spt (\varrho_0)\;\text{ such that }\;Z_t(x)\in \overline{B}_{r/2}\left(Z_t(x_0)\right)\quad \forall t\in [0,T].
\end{equation}

On the other hand, we have seen in Step 1 of the proof of Lemma~\ref{lem:Xinv} that \( |X_t- Z_t|\le O(\varepsilon)\) with \( O(\varepsilon) \) determined by \( M \), \( T \), \( L \), \( \varepsilon_0 \) only.
Thus, there exists \( \varepsilon_{*}>0 \) such that for all \( \varepsilon\in (0,\varepsilon_{*}) \), the following four properties hold:
\[
\varepsilon<1/L,\quad e^{MT}L\varepsilon<r/2,\quad \varepsilon\in (0,\varepsilon_1),\quad O(\varepsilon)\le r/2,
\]
where \( \varepsilon_1 \) is the constant from Lemma~\ref{lem:Xinv}.
By construction, such \( \varepsilon_{*} \) depends on \( M \), \( T \), \( r \), \( \kappa \), \( L \), \( \varepsilon_0 \) only.

In the rest of the proof, we assume that \( \varepsilon\in (0,\varepsilon_{*}) \).
In this case, thanks to Lemma~\ref{lem:Xinv}, the map \( X_t \) is a bi-Lipschitz homeomorphism.
Thus, the control \( u_t \) given by~\eqref{eq:control} is well-defined, Lipschitz and satisfies \( \spt(u_t) \subset \omega \) due to the choice of \( \alpha \).
Moreover, by construction, \( \mu^{u}_t=X_{t\sharp}\varrho_{\varepsilon} \) satisfies~\eqref{eq:conteq}.
By the choice of \( \varepsilon_{*} \), we have \( |X_t-Z_t| \le r/2\).
Hence
\[
\forall x\in \spt(\varrho_{\varepsilon})\quad \mbox{~~it holds~~} X_t(x)\in \overline{B}_{r/2}\left(Z_t(x)\right)\quad \forall t\in [0,T].
\]
By combining this statement with~\eqref{eq:middle}, we conclude that
\begin{equation}
  \label{eq:statement}
\forall x\in \spt(\varrho_{\varepsilon})\; \exists\, x_0\in \spt (\varrho_0)\;\text{ such that }\;X_t(x)\in \overline{B}_{r}\left(Z_t(x_0)\right)\quad \forall t\in [0,T].
\end{equation}

Now, it follows from Lemma~\ref{lem:common_time} that
for any \( x\in \spt(\varrho_{\varepsilon}) \) there exists an interval \( I(x)\subset (0,T) \) of length \( r/M \) such that \( X_t(x)\in \omega_0 \) for all \( t\in I(x) \).
By the definition of \( \alpha \), it holds
\[
\alpha(X_t(x))=a= -\kappa M \log \varepsilon /r\quad \forall x\in\spt(\varrho_{\varepsilon})\quad \forall t\in I(x),
\]
and therefore
\[
  e^{-\int_0^T\alpha(X_t(x))}\ge e^{-a\frac{r}{M}} = e^{\kappa\log\varepsilon} = \varepsilon^\kappa\quad \forall x\in\spt(\varrho_{\varepsilon}).
\]
By using \eqref{eq:X-Ynl2}, we have that \( \left(\int|X_T-Y_T|^2\,d\varrho_{\varepsilon}\right)^{1/2} \le C\varepsilon^{1+\kappa}\), for some \( C>0 \).
The estimate~\eqref{eq:epspower} follows from the second inequality in~\eqref{eq:W2ineq}.



We stress that \( C \) in~\eqref{eq:epspower} as well as \( \varepsilon_{*} \) and \( \kappa_{*}:=\frac{r}{TM} \) are determined by \( M \), \( T \), \( r \), \( L\), \( \varepsilon_0 \) only and~\eqref{eq:epspower} holds for each \( \varepsilon\in(0,\varepsilon_{*}) \).
Due to Lemma~\ref{lem:common_time}, the constant \( r \) depends on \( \omega \), \( \spt(\varrho_0) \) and \( M \) only.
Hence \( C \), \( \varepsilon_{*} \) and \( \kappa_{*} \) are determined by \( \omega \), \( \spt(\varrho_0) \), \( M \), \( T \), \( L \), \( \varepsilon_0 \) only.


\subsection{Eliminating Assumption \( (A_{3}) \)}
\label{subsec:A3}


In this section, we show that Assumption \( (A_{3}) \) can be removed.
Suppose that \( (A_{1,2}) \) hold and let \( V^{\sigma}(\cdot,\mu):= \eta^{\sigma}* V(\cdot,\mu) \) for all \( \mu\in \mathcal{P}_c(\mathbb{R}^d) \), where \( \sigma\in[0,1] \) and \(\eta^{\sigma}\) is the standard mollifier \eqref{eq:mollifier}.
For each \( \sigma\in [0,1] \), the vector field \( V^{\sigma} \) satisfies \( (A_{1}) \), with the same constant \( M \), and for each \( \sigma\in(0,1] \) it satisfies \( (A_3) \).

\begin{lemma}
Assumption \( (A_2) \) holds for \( V^{\sigma} \) if \( \sigma \) is sufficiently small.
\end{lemma}
\begin{proof}
Denote by \( \mu^{\sigma}_t \) the trajectory of~\eqref{eq:conteq-uncontrolled} corresponding to the nonlocal vector field \( V^{\sigma} \) and the initial condition \( \mu^{\sigma}_{0}=\varrho_0 \).
Thanks to~\cite[Theorem 1]{pogodaevNonlocalBalanceEquations2022}, the map \( \sigma\mapsto \mu^{\sigma}_{t} \) is continuous for each \( t\in [0,T] \).
Now, it follows from~\cite[Lemma 3]{pogodaevNonlocalBalanceEquations2022} that \( \Psi_t^{\sigma}\to \Psi_t \) pointwise as \( \sigma\to 0\), where \( \Psi^{\sigma}_t \) denotes the flow of \( (t,x)\mapsto V^{\sigma}_t(x,\mu^{\sigma}_t) \).
Since all \( \Psi^{\sigma}_t\) with \( \sigma\in [0,1] \), have the common Lipschitz constant \( e^{MT} \), the convergence is, in fact, uniform.
More precisely, letting \( K=\spt\varrho_0 \), we can say that \( \Psi_t^{\sigma}\to \Psi_t \) in \( C(K;\mathbb{R}^d) \) for each \( t\in [0,T] \).
Finally, recall that \( |V^{\sigma}|\le M \) due to \( (A_1) \).
Hence \( t\mapsto \Psi^{\sigma}_t(x) \) is \( M \)-Lipschitz for all \( x\in K \) and \( \sigma\in [0,1] \).
Therefore, \( \Psi^{\sigma}\to \Psi \) in \( C\left([0,T];C(K;\mathbb{R}^d)\right) \).
In particular, taking \( r \) from Lemma~\ref{lem:common_time}, we obtain
\
\begin{equation}
  \label{eq:FF}
|\Psi^{\sigma}_t(x) - \Psi_t(x)|\le r\quad \forall t\in [0,T]\quad \forall x\in K,
\end{equation}
if \(\sigma\) is sufficiently small.
Now, by Lemma~\ref{lem:common_time}, each curve \( t\mapsto \Psi^{\sigma}_t(x) \) satisfying~\eqref{eq:FF}, crosses \( \omega \) and thus \( (A_2) \) holds for all small \( \sigma>0 \).
\end{proof}

We now prove the main theorem under Assumption \( (A_{1,2}) \).

\begin{proofof}{Theorem~\ref{thm:main2}}
Before proceeding, we introduce some notation.
Denote by \( \mu_t[u,\vartheta] \) the solution of~\eqref{eq:conteq} corresponding to the admissible control \( u \) and the initial measure \( \vartheta \).
Similarly, \( \mu^{\sigma}_t[u,\vartheta] \) stands for the solution of~\eqref{eq:conteq}, with \( V \) replaced by \( V^{\sigma} \), which corresponds to the control \( u \) and the initial measure \( \vartheta \).

We know that \( V^{\sigma} \) satisfies the assumptions \( (A_{1,2,3}) \) for all small \( \sigma>0 \) and thus the corresponding continuity equation can be stabilized around a reference trajectory.
More precisely, for any perturbation \( \varrho\in \Pi_{\varepsilon_0,L}(\varrho_0) \), there exist \( \varepsilon_*,C>0 \) depending on \( M\), \( T \), \( L \), \( \kappa \), \( r \), \( \varepsilon_0 \), \( \spt(\varrho_0) \) only and such that, for all \( \varepsilon\in [0,\varepsilon_{*}] \), one can find an admissible control \( u^{\sigma,\varepsilon} \) satisfying
\[
\mathcal{W}_2\left(\mu_T^{\sigma}[u^{\sigma,\varepsilon},\varrho_{\varepsilon}],\mu_T^{\sigma}[0,\varrho_0]\right)\le C\varepsilon^{1+\kappa}.
\]
On the other hand,
\begin{align*}
  \mathcal{W}_2\left(\mu_T[u^{\sigma,\varepsilon},\varrho_{\varepsilon}],\mu_T[0,\varrho_0]\right)
  &\le
    \mathcal{W}_2\left(\mu_T[u^{\sigma,\varepsilon},\varrho_{\varepsilon}],\mu^\sigma_T[u^{\sigma,\varepsilon},\varrho_{\varepsilon}]\right)\\
  &+\mathcal{W}_2\left(\mu^{\sigma}_T[u^{\sigma,\varepsilon},\varrho_{\varepsilon}],\mu^\sigma_T[0,\varrho_0]\right)\\
  &+
\mathcal{W}_2\left(\mu^{\sigma}_T[0,\varrho_{0}],\mu_T[0,\varrho_{0}]\right).
\end{align*}
We know that the second term in the right-hand side is smaller than \( C\varepsilon^{1+\kappa} \).
Moreover, the Lipschitz constants of \( V^{\sigma} \) and \( u^{\sigma,\varepsilon} \) do not depend on \(\sigma\): by construction, they depend on \( M\), \( T \), \( L \), \( \varepsilon_0 \), \( \kappa \), \( r \), \( \spt(\varrho_0) \) only.
Therefore, by~\cite[Theorem 1]{pogodaevNonlocalBalanceEquations2022}, the first and the third terms converge to \( 0 \) as \( \sigma\to 0 \).
In particular, for each \( \varepsilon \) one can find \( \sigma=\sigma(\varepsilon) \) such that each of these terms is smaller than \( \varepsilon^{1+\kappa} \).
By letting \( u^{\varepsilon} := u^{\sigma(\varepsilon),\varepsilon} \), we obtain
\[
\mathcal{W}_2\left(\mu_T[u^{\varepsilon},\varrho_{\varepsilon}],\mu_T[0,\varrho_0]\right)\le (C+2)\varepsilon^{1+\kappa}.
\]
This completes the proof of Theorem~\ref{thm:main2} under the assumptions \( (A_{1,2}) \).
\end{proofof}

\subsection{Proof of Corollary~\ref{cor:main}}
\label{subsec:corollary}

In this section, we prove Corollary~\ref{cor:main}.
Our first goal is to understand how \( \lip(X_t) \) depends on \( \varepsilon \).
In view of the results of Section~\ref{subsec:A3}, we may assume, without loss of generality, that \( (A_3) \) holds.
Now, on \( \dom(DX_t) \), the derivative \( DX_t\) has the representation~\eqref{eq:DX_t}, i.e., \( DX_t=Q_tP_t \).
By arguing as in Step 3 from the proof of Lemma~\ref{lem:Xinv}, we deduce that the matrix \( Q_t \) satisfies \( \|Q_t\|\le O(1) \) on \( \dom(DX_t) \).
The inequalities~\eqref{eq:S_t_est} and~\eqref{eq:intalpha} allow us to estimate the second matrix \( P_t \):
\[
\|P_t\|\le 1 + O(-\log \varepsilon) + O(\varepsilon\log^2\varepsilon)=O(-\log \varepsilon) , \quad \forall \varepsilon\in (0,\varepsilon_1),\quad\forall t\in [0,T],
\]
where \( \varepsilon_1 \) is the constant from Lemma~\ref{lem:Xdiff}.
We then obtain
\begin{equation*}
\|DX_{t}\|_{\infty}\le O(-\log \varepsilon), \quad \forall \varepsilon\in (0,\varepsilon_1),\quad\forall t\in [0,T].
\end{equation*}
Since for any Lipschitz function \( f\colon \mathbb{R}^d\to \mathbb{R}^d \) it holds \( \lip(f)\le \|Df\|_{\infty} \), we get
\begin{equation}
  \label{eq:lipX}
\lip(X_t) \le O(-\log\varepsilon),\quad \forall \varepsilon\in (0,\varepsilon_1),\quad \forall t\in [0,T].
\end{equation}
Now, by the first inequality in~\eqref{eq:W2ineq}, it holds that
\begin{equation}
  \label{eq:keyestim}
  \mathcal{W}_2(X_{T\sharp}\nu_0,X_{T\sharp}\nu_1)\le O(-\log\varepsilon) \mathcal{W}_2(\nu_0,\nu_1),\quad \forall \nu_1,\nu_2\in \mathcal{P}_2(\mathbb{R}^{d}),\quad\forall \varepsilon\in (0,\varepsilon_1).
\end{equation}

We fix some \( \varepsilon\in (0,\varepsilon_{*}) \) and \( \kappa\in (0,\kappa_{*}) \), where \(\varepsilon_{*}\) and \(\kappa_{*}\) are the constants from Theorem~\ref{thm:main}, and let \( r(\varepsilon,\kappa)= -\varepsilon^{1+\kappa}/\log \varepsilon \).
By Theorem~\ref{thm:main}, any point of the set \( \Gamma_{\varepsilon}(\varrho_0) \) can be steered into the ball \( {\bf B}_{C\varepsilon^{1+\kappa}}(\mu_{T}) \) by a localized Lipschitz control.
Now, according to~\eqref{eq:keyestim}, there exists \( C_1>C \) such that any point of the open enlargement \({\bf B}_{r(\varepsilon,\kappa)} (\Gamma_{\varepsilon}(\varrho_0))\) can be steered into the ball \( {\bf B}_{C_1\varepsilon^{1+\kappa}}(\mu_T) \).
As before, the constant \( C_1 \) only depends on \( T \), \( M \), \( \omega \) and \( \spt(\varrho_0) \).
This observation completes the proof.

\subsection{Sharpness of the result}
\label{s-sharp}
In this section, we discuss the sharpness of the main results of this article, that are Theorem~\ref{thm:main} and Corollary \ref{cor:main}.


Recall that the sets of initial conditions \( \Gamma_{\varepsilon_{*}}(\varrho_0) \) and \( \Gamma^{\varepsilon_0,L}_{\varepsilon_{*}}(\varrho_0) \) considered before are completely defined by the corresponding perturbation classes.
In the first case, the perturbation class is composed of curves generated by flows of Lipschitz vector fields.
In the second case, it is composed of \emph{regular perturbations}.

Consider the general situation. We define
\[
\Gamma^{\Pi}_{\varepsilon_{*}}(\varrho_0) := \left\{\varrho_{\varepsilon}\text{ where } \varrho\in \Pi(\varrho_0),\; \varepsilon\in [0,\varepsilon_{*}]\right\},
\]
where \( \Pi(\varrho_0) \) is an arbitrary set of curves \( \varrho\colon [0,1]\to \mathcal{P}_c(\mathbb{R}^d) \) with the starting point \( \varrho_0 \).
Now, we can fix \( \Pi(\varrho_0) \) and ask whether the corresponding \( \Gamma^{\Pi}_{\varepsilon_{*}}(\varrho_0) \) can be \( \kappa \)-stabilized for some \( \kappa>0 \) and \( \varepsilon_{*}>0 \)?
In the following example, we show that such \( \Pi(\varrho_0) \) cannot be composed of all absolutely continuous curves originated from \( \varrho_0 \).

  \begin{example}

Fix \( d=1 \), \(V=0\), \(\omega=\mathbb R\), \(\varrho_0=\bm1_{[0,1]}\mathcal L^1\).
We define in \( \mathcal{P}_c(\mathbb{R}^d) \) a specific absolutely continuous curve $\varrho_t$, issuing from \( \varrho_0 \).
We set
\begin{equation}
  \label{eq:re}
        \varepsilon_{n} := \frac{1}{2^{n-1}},\quad
       \varrho_{\varepsilon_n} := \frac{1}{2^{n-1}}\sum_{k=1}^{2^{n-1}}\delta_{\frac{2k-1}{2^{n}}},\quad n\ge 1.
\end{equation}
     Note that \( \varrho_{\varepsilon_n} \) has \( 2^{n-1} \) atoms of mass \( 1/2^{n-1} \), and the distance between the closest atoms is \( 1/2^{n-1} \).     The optimal transport from \( \varrho_{\varepsilon_{n}} \) to \( \varrho_{\varepsilon_{n-1}} \) is given by merging two atoms of \( \varrho_{\varepsilon_{n}} \) into one. This completely describes the geodesic from \( \varrho_{\varepsilon_{n}} \) to \( \varrho_{\varepsilon_{n-1}} \), i.e. the curve realizing the optimal transport. The displacement of each atom is given by a distance \( 1/2^{n} \) (see Figure~\ref{fig:atoms}).
     Therefore,
     \[
       \mathcal{W}_2(\varrho_{\varepsilon_n},\varrho_{\varepsilon_{n-1}}) = \sqrt{2^{n-1}\cdot \frac{1}{2^{n-1}}\cdot\left(\frac{1}{2^{n}}\right)^2} = \frac{1}{2^{n}}.
     \]
     On the other hand, since the optimal transport map between \( \bm 1_{[a,b]} \mathcal{L}^1 \) and \( (b-a)\delta_{\frac{b-a}{2}} \) is \( \mathcal{T}\equiv \frac{b-a}{2} \), it holds
     \[
       \mathcal{W}_2(\varrho_{\varepsilon_n},\varrho_0) =\left(\sum_{k=1}^{2^{n-1}}\int_{\frac{2k-2}{2^{n}}}^{\frac{2k}{2^n}}\left|x-\frac{2k-1}{2^{n}}\right|^{2}\,d x\right)^{1/2}= \frac{1}{2^{n}\sqrt 3}=\frac{\varepsilon_n}{2\sqrt{3}}.
     \]
     We now define a curve \( \varrho\colon [0,1]\to \mathcal{P}_c(\mathbb{R})\ \) that satisfies $\varrho(0)=\varrho_0$, as well as $\varrho(\varepsilon_n)=\varrho_{\varepsilon_n}$, where \( \varrho_{\varepsilon_n} \) are defined by~\eqref{eq:re}. For each $t\in(\varepsilon_n,\varepsilon_{n-1})$, choose the reparametrized geodesic from $\varrho_{\varepsilon_n}$ to $\varrho_{\varepsilon_{n-1}}$ described above.  We stress that the curve $\varrho_t$ is then absolutely continuous as a function of its parameter, but that its points are never absolutely continuous measures, except for $t=0$.

     \begin{figure}[htb]
  \begin{center}
    \includegraphics[width=0.6\textwidth]{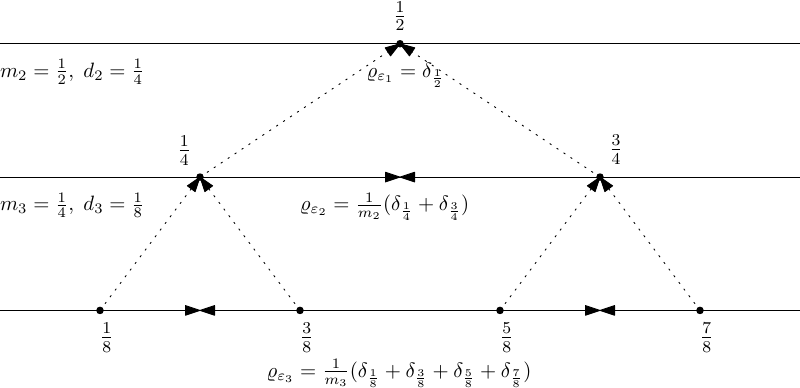}
\caption{Optimal transport between \( \varrho_{\varepsilon_3} \) and \( \varrho_{\varepsilon_2} \), then between \( \varrho_{\varepsilon_2} \) and \( \varrho_{\varepsilon_1} \). Here \( m_n=1/2^{n-1} \) is the mass of each atom of \( \varrho_{\varepsilon_n} \), \( d_n=1/2^n \) is the distance that each atom passes during the optimal mass transfer between \( \varrho_{\varepsilon_n} \) and \( \varrho_{\varepsilon_{n-1}} \).}
\label{fig:atoms}
\end{center}
\end{figure}

     In this way, the absolutely continuous curve $\varrho$ has length
     \[
       L(\varrho) = \sum_{n=1}^{\infty}\frac{1}{2^n} = 1 > \mathcal{W}_2(\varrho_0,\varrho_1)=\frac{1}{2\sqrt 3}.
     \]
    Lipschitz controls can only steer \(\varrho_{\varepsilon_{n}}\) to a measure \(\vartheta_n\) with the same atoms, but located in different positions.
    Such measures \( \vartheta_n \) satisfy \(\mathcal W_2(\varrho_0,\vartheta_n)\ge \frac{\varepsilon_n}{2\sqrt 3}\).
    Therefore, we have \( \lim_{n\to \infty} \varepsilon_n^{-1-\kappa}\mathcal{W}_2(\vartheta_n,\varrho_0) = +\infty  \) for any \( \kappa>0 \).
    Hence, \( \Gamma^{\Pi}_{\varepsilon_{*}}(\varrho_0) \) cannot be \( \kappa \)-stabilized, whatever \( \varepsilon_{*}>0 \) and \( \kappa>0 \) we choose.
    Similarily, any Wasserstein ball ${\bf B}_\eps(\varrho_0)$ can never be $\kappa$-stabilized, since it contains the measures $\varrho_{\varepsilon_n}$ for all sufficiently large $n$.
\end{example}

\subsection*{Acknowledgments:} We thank N. Gigli for helpful discussions about the topology of Wasserstein spaces.

\printbibliography

\end{document}